\documentclass{amsart}

\usepackage{amssymb}
\usepackage{amsthm}
\usepackage{amsmath}

\usepackage[foot]{amsaddr}

\usepackage{hyperref}

\theoremstyle{plain}
\newtheorem{proposition}{Proposition}[section]
\newtheorem{theorem}[proposition]{Theorem}
\newtheorem{lemma}[proposition]{Lemma}
\newtheorem{corollary}[proposition]{Corollary}
\theoremstyle{definition}

\newtheorem{definition}[proposition]{Definition}

\theoremstyle{remark}
\newtheorem{remark}[proposition]{Remark}
\newtheorem{conjecture}[proposition]{Conjecture}

\DeclareMathOperator{\Aff}{Aff}
\DeclareMathOperator{\Aut}{Aut}

\DeclareMathOperator{\Real}{Re}
\DeclareMathOperator{\Imaginary}{Im}

\DeclareMathOperator{\GL}{GL}

\DeclareMathOperator{\Hol}{Hol}

\DeclareMathOperator{\arcosh}{arcosh} 
\DeclareMathOperator{\arctanh}{tanh^{-1}}

\DeclareMathOperator{\Euc}{Euc} 
 
\DeclareMathOperator{\Wt}{wt} 
\DeclareMathOperator{\Id}{Id}

\DeclareMathOperator{\Cc}{\mathcal{C}}
\DeclareMathOperator{\Dc}{\mathcal{D}}
\DeclareMathOperator{\Ec}{\mathcal{E}}

\DeclareMathOperator{\Hc}{\mathcal{H}}

\DeclareMathOperator{\Lc}{\mathcal{L}}
\DeclareMathOperator{\Nc}{\mathcal{N}}
\DeclareMathOperator{\Oc}{\mathcal{O}}

\DeclareMathOperator{\Cb}{\mathbb{C}}

\DeclareMathOperator{\Nb}{\mathbb{N}}

\DeclareMathOperator{\Rb}{\mathbb{R}}
\DeclareMathOperator{\Xb}{\mathbb{X}}
\DeclareMathOperator{\Zb}{\mathbb{Z}}
\DeclareMathOperator{\Qb}{\mathbb{Q}}

\newcommand{\abs}[1]{\left|#1\right|}

\newcommand{\norm}[1]{\left\|#1\right\|}

\newcommand{\wh}[1]{\widehat{#1}}

\begin{document}

\title[Characterizing domains by their limit set]{Characterizing domains by the limit set of their automorphism group}
\author{Andrew M. Zimmer}\address{Department of Mathematics, University of Chicago, Chicago, IL 60637.}
\email{aazimmer@uchicago.edu}
\date{\today}
\subjclass[2010]{}

\begin{abstract} 
In this paper we study the automorphism group of smoothly bounded convex domains. We show that such a domain is biholomorphic to a ``polynomial ellipsoid'' (that is, a domain defined by a weighted homogeneous balanced polynomial) if and only if the limit set of the automorphism group intersects at least two closed complex faces of the set. The proof relies on a detailed study of the geometry of the Kobayashi metric and ideas from the theory of non-positively curved metric spaces. We also obtain a number of other results including the Greene-Krantz conjecture in the case of uniform non-tangential convergence, new results about continuous extensions (of biholomorphisms and complex geodesics), and a new Wolff-Denjoy theorem.  \end{abstract} 
\maketitle

\section{Introduction}

Given a bounded domain $\Omega \subset \Cb^d$ let $\Aut(\Omega)$ be the automorphism group of $\Omega$, that is the group of biholomorphisms of $\Omega$. The group $\Aut(\Omega)$ is a Lie group and acts properly on $\Omega$. When $\partial \Omega$ has nice properties there are believed to be few domains with large automorphism group. For instance:

\begin{theorem}[Wong and Rosay Ball Theorem \cite{R1979, W1977}]
Suppose $\Omega \subset \Cb^d$ is a bounded strongly pseudoconvex domain. Then $\Aut(\Omega)$ is non-compact if and only if $\Omega$ is biholomorphic to the unit ball. 
\end{theorem}

Other bounded domains with large automorphism group and smooth boundary can be constructed using weighted homogeneous balanced polynomials. Given $m=(m_1, \dots, m_d) \in \Zb^d_{>0}$ define an associated weight function $\Wt_m: \Zb_{\geq 0}^d \rightarrow \Qb$ by 
\begin{align*}
\Wt_m(\alpha) := \sum_{i=1}^d \frac{ \alpha_i}{2 m_i}.
\end{align*}
A real polynomial $p:\Cb^d \rightarrow \Rb$ is then called a \emph{weighted homogeneous balanced polynomial} if there exists some $m \in \Zb_{>0}^d$ so that 
\begin{align*}
p(z_1, \dots, z_d) = \sum_{\Wt_m(\alpha)=\Wt_m(\beta)=1/2} C_{\alpha, \beta} z^{\alpha} \overline{z}^{\beta}.
\end{align*}

\begin{definition} A domain $\Ec$ is called a \emph{polynomial ellipsoid} if 
\begin{align*}
\Ec = \left\{ (w,z) \in \Cb \times \Cb^d : \abs{w}^2 + p(z) < 1\right\}
\end{align*}
where $p:\Cb^d \rightarrow \Rb$ is a weighted homogeneous balanced polynomial.
\end{definition}

There are many examples of bounded polynomial ellipsoid, for instance 
\begin{align*}
\left\{( w,z) \in \Cb \times \Cb : \abs{w}^2 + \abs{z}^{2m} < 1\right\}
\end{align*}
for any integer $m > 1$. Moreover, a polynomial ellipsoid always has a non-compact automorphism group: when $p$ is a weighted homogeneous balanced polynomial the domain
\begin{align*}
\Cc: = \left\{ (w,z) \in \Cb \times \Cb^{d} : \Imaginary(w) > p(z)\right\}
\end{align*}
has non-compact automorphism group (namely real translations in the first variable and a dilation) and the map 
given by 
\begin{align*}
F(z_0,\dots, z_d) = \left( \frac{1-iz_0/4}{1+iz_0/4}, \frac{z_1}{(1+iz_0/4)^{1/m_1}}, \dots, \frac{z_d}{(1+iz_0/4)^{1/m_d}}\right)
\end{align*}
is a biholomorphism of $\Cc$ to 
\begin{align*}
\Ec = \left\{ (w,z) \in \Cb \times \Cb^d : \abs{w}^2 + p(z) < 1\right\}.
\end{align*}

In a series of papers Bedford and Pinchuk~\cite{BP1988, BP1991, BP1994, BP1998} studied the automorphism group of domains of finite type and in particular gave the following characterization of the domains described above:

\begin{theorem}\label{thm:BP}\cite{BP1994}
Suppose $\Omega$ is a bounded convex domain with $C^\infty$ boundary and finite type in the sense of D'Angelo. Then $\Aut(\Omega)$ is non-compact if and only if $\Omega$ is biholomorphic to a polynomial ellipsoid. 
\end{theorem}

There are many other results characterizing special domains via the properties of their automorphism group and boundary, see for instance~\cite{GK1987, K1992, W1995, G1997, V2009} and the survey paper~\cite{IK1999}. Like the two theorems mentioned above, almost all previous work assumes that either the entire boundary or a point in the limit set satisfies some infinitesimal condition (for instance strong pseudoconvexity, finite type, or Levi flat). In contrast to these result we provide a new characterization of balanced domains in terms of the geometry of the limit set. 

We define the \emph{limit set} of $\Omega$, denoted $\Lc(\Omega)$,  to be the set of points $x \in \partial \Omega$ where there exists some $p \in \Omega$ and some sequence $\varphi_n \in \Aut(\Omega)$ such that $\varphi_n p \rightarrow x$. Since $\Aut(\Omega)$ acts properly on $\Omega$, when $\Aut(\Omega)$ is non-compact the limit set $\Lc(\Omega)$ is non-empty. 

 If $\Omega$ is a bounded convex domain with $C^1$ boundary and $x \in \partial \Omega$ let $T_x^{\Cb} \partial \Omega \subset \Cb^d$ be the complex hyperplane tangent to $\partial \Omega$ at $x$. Then the \emph{closed complex face} of a point $x \in \partial \Omega$ is the closed set  $\partial \Omega \cap T_x^{\Cb} \partial \Omega$.

With this language we will prove:

\begin{theorem}\label{thm:main}
Suppose $\Omega$ is a bounded convex domain with $C^\infty$ boundary. Then the following are equivalent:
\begin{enumerate}
\item $\Lc(\Omega)$ intersects at least two closed complex faces of $\partial \Omega$, 
\item $\Omega$ is biholomorphic to a polynomial ellipsoid. 
\end{enumerate}
\end{theorem}

\begin{remark} Suppose $\Omega$ is a bounded convex domain with $C^1$ boundary and $x,y \in \partial \Omega$ are distinct. Then the following are equivalent: 
\begin{enumerate}
\item $x$ and $y$ are in different closed complex faces of $\partial \Omega$, 
\item $T_x^{\Cb} \partial \Omega \neq T_y^{\Cb} \partial \Omega$,
\item the complex line containing $x$ and $y$ intersects $\Omega$.
\end{enumerate}
\end{remark} 

\subsection*{Acknowledgments} 

I would like to thank Anders Karlsson for bringing to my attention his work in~\cite{K2005b}, in particular the general theory of isometries of metric spaces described there motivated some of the results below. I would also like to thank Eric Bedford, Gautam Bharali, and a referee for helpful comments which improved the exposition of this paper. This material is based upon work supported by the National Science Foundation under Grant Number NSF 1400919.

\section{Outline of the proof of Theorem~\ref{thm:main}}\label{sec:outline}

One of the key ideas in the proof of Theorem~\ref{thm:main} is to show that a smoothly bounded convex domain endowed with its the Kobayashi metric behaves like a Gromov hyperbolic metric space. In this section we will recall some properties of Gromov hyperbolic metric spaces, describe analogues of these properties for the Kobayashi metric on a smoothly bounded convex domain, and then describe the main steps in the proof Theorem~\ref{thm:main}. We then end this section with some other applications of these negative curvature type properties.
 
 \subsection{Gromov hyperbolic metric spaces} Suppose $(X,d)$ is a metric space. If $I \subset \Rb$ is an interval, a curve $\sigma: I \rightarrow X$ is a \emph{geodesic} if $d(\sigma(t_1),\sigma(t_2)) = \abs{t_1-t_2}$ for all $t_1, t_2 \in I$.  A \emph{geodesic triangle} in a metric space is a choice of three points in $X$ and geodesic segments  connecting these points. A geodesic triangle is said to be \emph{$\delta$-thin} if any point on any of the sides of the triangle is within distance $\delta$ of the other two sides. 

\begin{definition}
A proper geodesic metric space $(X,d)$ is called \emph{$\delta$-hyperbolic} if every geodesic triangle is $\delta$-thin. If $(X,d)$ is $\delta$-hyperbolic for some $\delta\geq0$, then $(X,d)$ is called \emph{Gromov hyperbolic}.
\end{definition}

\begin{remark} Bridson and Haefliger's~\cite{BH1999} book on non-positively curved metric spaces is one of the standard references for Gromov hyperbolic metric spaces. \end{remark}

For a metric space $(X,d)$ the \emph{Gromov product} of $p,q \in X$ at $o \in X$ is defined to be:
\begin{align*}
(p|q)_o = \frac{1}{2} ( d(p,o) + d(o,q) - d(p,q) ).
\end{align*}
If $(X,d)$ is Gromov hyperbolic, then one can use the Gromov product to define an abstract boundary $X(\infty)$ called the \emph{ideal boundary}. A sequence $(p_i)_{i\geq1} \subset X$ is said to be \emph{converge to infinity} if 
\begin{align*}
\lim_{i,j \rightarrow \infty} (p_i|p_j)_o = \infty
\end{align*}
 for some (and hence any) $o \in X$. The set $X(\infty)$ is then the equivalence classes of sequences converging to infinity where two such sequences $(p_i)_{i \geq 1}$ and $(q_j)_{j \geq 1}$ are equivalent if 
\begin{align*}
\lim_{i,j \rightarrow \infty} (p_i|q_j)_o = \infty
\end{align*}
for some (and hence any) $o \in X$. Finally, there is a natural topology on $X \cup X(\infty)$ which makes it a compactification of $X$ (see for instance Chapter III.H Section 3 in~\cite{BH1999}). It is important to note that when $(X,d)$ is not Gromov hyperbolic, the relation defined above may not be transitive. 

This compactification behaves nicely with respect to 1-Lipschitz maps $f:X \rightarrow X$. In particular, Karlsson proved the following Wolff-Denjoy theorem:

\begin{theorem}\label{thm:gromov_wd}\cite[Proposition 5.1]{K2001} Suppose $(X,d)$ is Gromov hyperbolic. If $f: X \rightarrow X$ is 1-Lipschitz, then either:
\begin{enumerate}
\item for every $p \in X$ the orbit $\{ f^n(p) : n \in \Nb\}$ is bounded in $(X,d)$, 
\item there exists some $x \in X(\infty)$ so that 
\begin{align*}
\lim_{n \rightarrow \infty} f^n(p) = x
\end{align*}
for all $p \in X$. 
\end{enumerate}
\end{theorem}

For Gromov hyperbolic metric spaces, it is possible to characterize the isometries in terms of their long term behavior: 

\begin{definition}
Suppose $(X,d)$ is Gromov hyperbolic and $\varphi: X \rightarrow X$ is an isometry. Then:
\begin{enumerate}
\item $\varphi$ is \emph{elliptic} if the orbit $\{ \varphi^n(p) : n \in \Zb\}$ is bounded for some (hence any) $p \in X$, 
\item $\varphi$ is \emph{hyperbolic} if $\phi$ is not elliptic and 
\begin{align*}
\lim_{n \rightarrow \infty} \varphi^n(p) \neq \lim_{n \rightarrow -\infty} \varphi^n(p) 
\end{align*}
for some (hence any) $p \in X$, 
\item $\phi$ is \emph{parabolic} if $\varphi$ is not elliptic and 
\begin{align*}
\lim_{n \rightarrow \infty} \varphi^n(p) = \lim_{n \rightarrow -\infty} \varphi^n(p) 
\end{align*}
for some (hence any) $p \in X$.
\end{enumerate}
\end{definition}

\begin{remark} Notice that Theorem~\ref{thm:gromov_wd} implies that every isometry of $(X,d)$ is either elliptic, hyperbolic, or parabolic. \end{remark}

One more important property of Gromov hyperbolic metric spaces is that geodesics joining two distinct points in the ideal boundary ``bend'' into the space:
\begin{theorem}\label{thm:gromov_visible}
Suppose $(X,d)$ is Gromov hyperbolic. If $x,y \in X(\infty)$ and $V_x, V_y$ are neighborhoods of $x,y$ in $X \cup X(\infty)$ so that $\overline{V_x} \cap \overline{V_y} = \emptyset$, then there exists a compact set $K \subset X$ with the following property: if $\sigma: [0,T] \rightarrow X$ is a geodesic with $\sigma(0) \in V_x$ and $\sigma(T) \in V_y$,  then $\sigma \cap K \neq \emptyset$. 
\end{theorem}

\begin{remark} Conditions of this type were first introduced by Eberlein and O'Neill~\cite{EO1973} in the context of non-positively curved simply connected Riemannian manifolds. Also see~\cite[page 54]{BGS1985} or~\cite[page 294]{BH1999}.
\end{remark}

\subsection{The Kobayashi metric on smoothly bounded convex domains}\label{subsec:neg_curv_kob} 

We now describe how the properties described above extend to the Kobayashi metric on a smoothly bounded convex domain. 

Given a domain $\Omega \subset \Cb^d$, let $K_\Omega$ be the Kobayashi distance on $\Omega$. We recently proved the following:

\begin{theorem}\cite{Z2014} Suppose $\Omega \subset \Cb^d$ is a bounded convex domain with $C^\infty$ boundary. Then the following are equivalent: 
\begin{enumerate}
\item $\Omega$ has finite type in the sense of D'Angelo, 
\item $(\Omega, K_\Omega)$ is Gromov hyperbolic. 
\end{enumerate}
\end{theorem}

\begin{remark} Balogh and Bonk~\cite{BB2000} proved that the Kobayashi metric on a strongly pseudoconvex domain is Gromov hyperbolic. \end{remark}

For convex domains of finite type, the ideal and topological boundary also coincide:

\begin{proposition}\cite[Proposition 11.3, Proposition 11.5]{Z2014} Suppose $\Omega \subset \Cb^d$ is a bounded convex domain with finite type in the sense of D'Angelo. If $p_n, q_m \in \Omega$ are sequences such that $p_n \rightarrow x \in \partial \Omega$ and $q_m \rightarrow y \in \partial \Omega$, then 
\begin{align*}
\lim_{n, m \rightarrow \infty} (p_n|q_m)_o= \infty \text{ if and only if } x=y.
\end{align*}
In particular, the ideal boundary of $(\Omega, K_\Omega)$ is homeomorphic to the topological boundary of $\Omega$.
\end{proposition} 

One important step in the proof of Theorem~\ref{thm:main} is to show that the Gromov product is still reasonably behaved even when the domain does not have finite type:

\begin{theorem}\label{thm:gromov_prod_i}(see Theorem~\ref{thm:gromov_prod} below)
Suppose $\Omega \subset \Cb^d$ is a bounded convex domain with $C^{1,\alpha}$ boundary and $p_n, q_m \in \Omega$ are sequences such that $p_n \rightarrow x \in \partial \Omega$ and $q_m \rightarrow y \in \partial \Omega$. 
\begin{enumerate}
\item If $x=y$, then 
\begin{align*}
\lim_{n,m \rightarrow \infty} ( p_n | q_m)_o = \infty.
\end{align*}
\item If 
\begin{align*}
\limsup_{n,m \rightarrow \infty} \ ( p_n | q_m)_o = \infty,
\end{align*}
then $T_{x}^{\Cb} \partial \Omega = T_{y}^{\Cb} \partial \Omega$.
\end{enumerate}
\end{theorem}

Although this behavior is much weaker than the finite type case, we can still use Theorem~\ref{thm:gromov_prod_i} to prove variants of Theorem~\ref{thm:gromov_wd} and Theorem~\ref{thm:gromov_visible}. For instance, Theorem~\ref{thm:gromov_prod_i} and Karlsson's~\cite{K2001} work on the behavior of 1-Lipschitz maps on general metric spaces imply the following:

\begin{theorem}\label{thm:wolf_i}(see Theorem~\ref{thm:wolf} below)
 Suppose $\Omega \subset \Cb^d$ is a bounded convex domain with $C^{1,\alpha}$ boundary. If $f:\Omega \rightarrow \Omega$ is 1-Lipschitz with respect to the Kobayashi metric, then $f$ either has a fixed point in $\Omega$ or there exists some $x \in \partial \Omega$ so that 
 \begin{align*}
 \lim_{k \rightarrow \infty} d_{Euc}(f^{k}(p) , T_x^{\Cb} \partial \Omega) = 0
 \end{align*}
 for all $p \in \Omega$. 
 \end{theorem}
 
 \begin{remark}
 Abate and Raissy~\cite{AR2014} proved Theorem~\ref{thm:wolf_i} with the additional assumption that $\partial \Omega$ is $C^2$.  
  \end{remark}

Using Theorem~\ref{thm:wolf_i} we can characterize the automorphisms of $\Omega$ into elliptic, hyperbolic, and parabolic elements. Suppose $\Omega \subset \Cb^d$ is a bounded convex domain with $C^{1,\alpha}$ boundary and $\varphi \in \Aut(\Omega)$. Then by Theorem~\ref{thm:wolf_i} either $\varphi$ has a fixed point in $\Omega$ or there exists a complex supporting hyperplane $H_{\varphi}^+$ of $\Omega$ so that 
\begin{align*}
 \lim_{k \rightarrow \infty} d_{Euc}(\varphi^k p , H_\varphi^+) = 0
 \end{align*}
 for all $p \in \Omega$. In this latter case, we call $H_{\varphi}^+$ the \emph{attracting hyperplane of} $\varphi$. 
 
  \begin{definition}
Suppose $\Omega \subset \Cb^d$ is a bounded convex domain with $C^{1,\alpha}$ boundary and $\varphi \in \Aut(\Omega)$. Then:
\begin{enumerate}
\item $\varphi$ is \emph{elliptic} if $\varphi$ has a fixed point in $\Omega$, 
\item $\varphi$ is \emph{parabolic} if $\varphi$ has no fixed point in $\Omega$ and $H_{\varphi}^+ = H_{\varphi^{-1}}^+$,
\item $\varphi$ is \emph{hyperbolic} if $\varphi$ has no fixed points in $\Omega$ and $H_{\varphi}^+ \neq H_{\varphi^{-1}}^+$. In this case we call $H_{\varphi}^- : = H_{\varphi^{-1}}^+$  the \emph{repelling hyperplane of} $\varphi$.
\end{enumerate}
\end{definition}

\begin{remark} Notice that Theorem~\ref{thm:wolf_i} implies that every automorphism of $\Omega$ is either elliptic, hyperbolic, or parabolic. \end{remark}

We can also use Theorem~\ref{thm:gromov_prod_i} to establish a version of Theorem~\ref{thm:gromov_visible} for convex domains:

\begin{theorem}\label{thm:visible}(see Theorem~\ref{thm:geod_converge} below)
 Suppose $\Omega \subset \Cb^d$ is a bounded convex domain with $C^{1,\alpha}$ boundary. If $x,y \in \partial \Omega$ and $V_x, V_y$ are neighborhoods of $T_x^{\Cb} \partial \Omega \cap \partial \Omega$, $T_y^{\Cb} \partial \Omega \cap \partial \Omega$ in $\overline{\Omega}$ so that $\overline{V_x} \cap \overline{V_y} = \emptyset$, then there exists a compact set $K \subset \Omega$ with the following property: if $\sigma: [0,T] \rightarrow (\Omega,K_\Omega)$ is a geodesic with $\sigma(0) \in V_x$ and $\sigma(T) \in V_y$,  then $\sigma \cap K \neq \emptyset$. 
\end{theorem}

\subsection{Outline of the proof of Theorem~\ref{thm:main}}

The difficult direction of Theorem~\ref{thm:main} is to show that $\Omega$ is biholomorphic to a polynomial ellipsoid when the limit set intersects at least two different closed complex faces. The main steps in the proof of this direction are the following: 
\begin{enumerate}
\item Use the metric properties described in Subsection~\ref{subsec:neg_curv_kob} to show that $\Aut(\Omega)$ contains a hyperbolic element $\varphi$ and an orbit $\{\varphi^n(o) : n \in \Nb\}$ of this element converges non-tangentially to a boundary point $x^+_{\varphi}$ (see Sections~\ref{sec:axial_autos} and~\ref{sec:finding_hyp}).
\item Use a rescaling argument and the metric properties described in Subsection~\ref{subsec:neg_curv_kob} to show that $x^+_{\varphi}$ has finite type in the sense of D'Angelo (see Section~\ref{sec:greene_krantz}).
\item Use another rescaling argument to show that the entire boundary has finite type (see Section~\ref{sec:entire_bd_finite_type}).
\item Apply Bedford and Pinchuk's result to deduce that $\Omega$ is biholomorphic to a polynomial ellipsoid. 
\end{enumerate}

\subsection{Other applications of Theorem~\ref{thm:gromov_prod_i}} In this subsection we describe some other applications of Theorem~\ref{thm:gromov_prod_i} (many of which are used in the proof of Theorem~\ref{thm:main}). 

\subsubsection{Boundary extensions} A convex domain $\Omega \subset \Cb^d$ is called \emph{$\Cb$-strictly convex} if every supporting complex hyperplane intersects $\partial \Omega$ at exactly one point. When $\partial \Omega$ is $C^1$ this is equivalent to $T_{x}^{\Cb} \partial \Omega \cap \partial \Omega = \{x\}$ for every $x \in \partial \Omega$. For $\Cb$-strictly convex domains we have the following corollary of Theorem~\ref{thm:gromov_prod_i}:

\begin{corollary}
Suppose $\Omega \subset \Cb^d$ is a bounded $\Cb$-strictly convex domain with $C^{1,\alpha}$ boundary. If $p_n, q_m \in \Omega$ are sequences such that $p_n \rightarrow x \in \partial \Omega$ and $q_m \rightarrow y \in \partial \Omega$, then 
\begin{align*}
\lim_{n, m \rightarrow \infty} (p_n|q_m)_o= \infty \text{ if and only if } x=y.
\end{align*}
\end{corollary}

As an application of this corollary we will prove the following result about boundary extensions:

\begin{theorem}\label{thm:cont_ext_i} (see Theorem~\ref{thm:cont_ext} below)\label{thm:cont_ext_i} Suppose $\Omega_1 \subset \Cb^{d_1}$ and $\Omega_2 \subset \Cb^{d_2}$ are bounded convex domains with $C^{1,\alpha}$ boundaries. If $\Omega_2$ is $\Cb$-strictly convex, then every isometric embedding $f:(\Omega_1,K_{\Omega_1}) \rightarrow (\Omega_2,K_{\Omega_2})$ extends to a  continuous map $\overline{f}: \overline{\Omega_1} \rightarrow \overline{\Omega_2}$. 
\end{theorem}

In particular:

\begin{corollary}
Suppose $\Omega_1, \Omega_2 \subset \Cb^{d}$ are bounded convex domains with $C^{1,\alpha}$ boundaries. If $\Omega_2$ is $\Cb$-strictly convex, then every biholomorphism $f:\Omega_1 \rightarrow \Omega_2$ extends to a continuous map $\overline{f}: \overline{\Omega_1} \rightarrow \overline{\Omega_2}$. 
\end{corollary}

\begin{corollary}
Suppose $\Omega \subset \Cb^{d}$ is a bounded $\Cb$-strictly convex domain with $C^{1,\alpha}$ boundary. Then every complex geodesic $\varphi: \Delta \rightarrow \Omega$ extends to a continuous map $\overline{\varphi}: \overline{\Delta} \rightarrow \overline{\Omega}$. 
\end{corollary}

\begin{remark}\ 
\begin{enumerate}
\item As mentioned above, Theorem~\ref{thm:cont_ext_i} is a consequence of the behavior of the Gromov product on convex domains. It is worth mentioning that Forstneri{\v{c}} and Rosay~\cite[Theorem 1.1]{FR1987} also used the behavior of the Gromov product (without using this terminology) to establish continuous boundary extensions of proper holomorphic maps between strongly pseudoconvex domains.
\item There is also an alternative approach to proving boundary extensions of holomorphic maps which uses lower bounds on the infinitesimal Kobayashi metric and a Hardy-Littlewood type lemma, see for instance~\cite{DF1979, CHL1988, M1993, B2016}. These arguments only appear to work when the infinitesimal Kobayashi metric obeys some estimate of the form 
\begin{align*}
k_\Omega(x;v) \geq \frac{\norm{v}}{f(d_{\Euc}(x,\partial\Omega))}
\end{align*}
where $\int_0^\epsilon \frac{f(r)}{r} dr < \infty$, see~\cite{BZ2016}. It is possible to construct smoothly bounded $\Cb$-strictly convex domains where the infinitesimal Kobayash metric fails to satisfy such estimates.
\end{enumerate}
\end{remark}

\subsubsection{Non-existence of holomorphic maps} We will also use Theorem~\ref{thm:gromov_prod} to show that certain holomorphic maps $f:\Delta \times \Delta \rightarrow \Omega$ cannot exist:

 \begin{theorem}\label{thm:prod_map_i}(see Theorem~\ref{thm:prod_map} below) 
 Suppose $\Omega \subset \Cb^d$ is a bounded convex domain with $C^{1,\alpha}$ boundary. Then there does not exist a holomorphic map $f: \Delta \times \Delta \rightarrow \Omega$ which induces an isometric embedding $(\Delta \times \Delta, K_{\Delta \times \Delta}) \rightarrow (\Omega, K_{\Omega})$.
 \end{theorem}

\begin{remark} \ \begin{enumerate}
\item If $(X,d)$ is a Gromov hyperbolic metric space then there does not exist an isometric embedding of $(\Rb^2, d_{\Euc})$ into $(X,d)$. In particular, the above Corollary shows that convex domains with $C^{1,\alpha}$ boundary have some hyperbolic behavior. 
\item The statement of Theorem~\ref{thm:prod_map} below is considerably more general and is used to prove a special case of the Greene-Krantz conjecture (see Theorem~\ref{thm:unif_tang_conv_i} below).
\end{enumerate}
\end{remark}

\subsubsection{The Greene-Krantz conjecture} The second main step in the proof of Theorem~\ref{thm:main} is related to an old conjecture of Greene and Krantz. In particular, in the 1990's Greene and Krantz conjectured:

\begin{conjecture}\cite{GK1993} Suppose that $\Omega$ is a bounded pseudoconvex domain with $C^\infty$ boundary. If $x \in \Lc(\Omega)$, then $x$ has finite type in the sense of Kohn/D'Angelo/Catlin. 
\end{conjecture}

There are a number of partial results supporting the conjecture, see for instance the survey paper~\cite{K2013}. In Section~\ref{sec:greene_krantz}, we will prove the following special case of this conjecture: 

 \begin{theorem}\label{thm:unif_tang_conv_i}(see Theorem~\ref{thm:unif_tang_conv} below)
Suppose $\Omega \subset \Cb^d$ is a bounded convex domain with $C^\infty$ boundary. If there exists $o \in \Omega$, $x \in \partial \Omega$, $M \geq 0$, and $T \in \Rb$ so that 
\begin{align*}
\{ x + e^{-2t} n_{x} : t > T \}  \subset \Aut(\Omega) \cdot \{ p \in \Omega: K_\Omega(p,o) \leq M\},
\end{align*}
then $x$ has finite type in the sense of D'Angelo. 
\end{theorem}

Here is the idea of the proof: if  $x$ had infinite type, then we could use a rescaling argument to construct a holomorphic map $f:\Delta \times \Delta \rightarrow \Omega$ having (essentially) the properties in the hypothesis of Theorem~\ref{thm:prod_map_i} which is impossible.

\section{Preliminaries}\label{sec:prelim}

\subsection{Notations}

\begin{enumerate}
\item For $z \in \Cb^d$ let $\norm{z}$ be the standard Euclidean norm and $d_{\Euc}(z_1, z_2) = \norm{z_1-z_2}$ be the standard Euclidean distance. 
\item Given an open set $\Omega \subset \Cb^d$, $p \in \Omega$, and $v \in \Cb^d \setminus \{0\}$ let 
\begin{align*}
\delta_{\Omega}(p)= \inf \{ d_{\Euc}(p,x) : x\in \partial \Omega \}
\end{align*}
and 
\begin{align*}
\delta_{\Omega}(p;v)= \inf \{ d_{\Euc}(p,x) : x\in \partial \Omega \cap (p+\Cb \cdot v) \}.
\end{align*}
\end{enumerate}

\subsection{The Kobayashi metric}

Given a domain $\Omega \subset \Cb^d$ the \emph{(infinitesimal) Kobayashi metric} is the pseudo-Finsler metric
\begin{align*}
k_{\Omega}(x;v) = \inf \left\{ \abs{\xi} : f \in \Hol(\Delta, \Omega), \ f(0) = x, \ d(f)_0(\xi) = v \right\}.
\end{align*}
By a result of Royden~\cite[Proposition 3]{R1971} the Kobayashi metric is an upper semicontinuous function on $\Omega \times \Cb^d$. In particular if $\sigma:[a,b] \rightarrow \Omega$ is an absolutely continuous curve (as a map $[a,b] \rightarrow \Cb^d$), then the function 
\begin{align*}
t \in [a,b] \rightarrow k_\Omega(\sigma(t); \sigma^\prime(t))
\end{align*}
is integrable and we can define the \emph{length of $\sigma$} to  be
\begin{align*}
\ell_\Omega(\sigma)= \int_a^b k_\Omega(\sigma(t); \sigma^\prime(t)) dt.
\end{align*}
One can then define the \emph{Kobayashi pseudo-distance} to be
\begin{multline*}
 K_\Omega(x,y) = \inf \left\{\ell_\Omega(\sigma) : \sigma\colon[a,b]
 \rightarrow \Omega \text{ is absolutely continuous}, \right. \\
 \left. \text{ with } \sigma(a)=x, \text{ and } \sigma(b)=y\right\}.
\end{multline*}
This definition is equivalent to the standard definition by a result of Venturini~\cite[Theorem 3.1]{V1989}.

A nice introduction to the Kobayashi metric and its properties can be found in~\cite{K2005} or~\cite{A1989}. 

One important property of the Kobayashi metric on a convex set is the following:

\begin{proposition}\cite{B1980}\label{prop:completeness}
Suppose $\Omega \subset \Cb^d$ is a convex domain Then the following are equivalent:
\begin{enumerate}
\item $(\Omega, K_\Omega)$ is a Cauchy complete geodesic metric space, 
\item $\Omega$ does not contain any complex affine lines.
\end{enumerate}
\end{proposition}

\subsection{The disk and the upper half plane}

For the disk and upper half plane the Kobayashi metric coincides with the Poincar{\'e} metric. 

Let $\Delta = \{ z \in \Cb : \abs{z} < 1\}$. Then 
\begin{align*}
k_{\Delta}(\zeta; v) = \frac{\abs{v}}{1-\abs{\zeta}^2}\end{align*}
and 
\begin{align*}
K_{\Delta}(\zeta_1,\zeta_2) = \arctanh \abs{\frac{\zeta_1-\zeta_2}{1-\zeta_1\overline{\zeta_2}}}.
\end{align*}
Next let $\Hc = \{ z \in \Cb : \Imaginary(z) >0\}$. Then 
\begin{align*}
k_{\Hc}(\zeta; v) = \frac{\abs{v}}{2\Imaginary(\zeta)}
\end{align*}
and 
\begin{align*}
K_{\Hc}(\zeta_1,\zeta_2) = \frac{1}{2} \arcosh \left( 1+ \frac{\abs{\zeta_1-\zeta_2}^2}{2\Imaginary(\zeta_1)\Imaginary(\zeta_2) } \right).
\end{align*}

\subsection{Almost geodesics}

In the proof of Theorem~\ref{thm:main} it will often be convenient to work with a class of curves which we call almost-geodesics:

\begin{definition}
Suppose $(X,d)$ is a metric space and $I \subset \Rb$ is an interval.
\begin{enumerate}
\item  A curve $\sigma: I \rightarrow X$ is a \emph{$(A,B)$-quasi-geodesic} if 
\begin{align*}
\frac{1}{A} \abs{t-s} - B \leq d(\sigma(s), \sigma(t)) \leq A\abs{t-s} + B
\end{align*}
for all $s,t \in I$. 
\item If $K \geq 1$ then a curve $\sigma: I \rightarrow X$ is an \emph{$K$-almost-geodesic} if $\sigma$ is an $(1,\log K)$-quasi-geodesic and 
\begin{align*}
d(\sigma(s), \sigma(t)) \leq K\abs{t-s} 
\end{align*}
for all $s,t \in I$. 
\end{enumerate}
\end{definition}

The main motivation for considering almost-geodesics is Proposition~\ref{prop:rough_geodesics} below which shows that inward pointing normal lines can be parametrized to be an almost-geodesics for convex domains with $C^{1,\alpha}$ boundary. We should also note that an $1$-almost-geodesic is a geodesic, thus motivating the choice of $\log K$ additive factor. 

\section{The Gromov product}

In this section we prove Theorem~\ref{thm:gromov_prod_i} which we restate:

\begin{theorem}\label{thm:gromov_prod}
Suppose $\Omega \subset \Cb^d$ is a bounded convex domain with $C^{1,\alpha}$ boundary and $p_n, q_m \in \Omega$ are sequences such that $p_n \rightarrow x \in \partial \Omega$ and $q_m \rightarrow y \in \partial \Omega$. 
\begin{enumerate}
\item If $x=y$, then 
\begin{align*}
\lim_{n,m \rightarrow \infty} ( p_n | q_m)_o= \infty.
\end{align*}
\item If 
\begin{align*}
\limsup_{n,m \rightarrow \infty} \ ( p_n | q_m)_o = \infty,
\end{align*}
then $T_{x}^{\Cb} \partial \Omega = T_{y}^{\Cb} \partial \Omega$.
\end{enumerate}
\end{theorem}

 We begin by proving a series of lemmas. 

\begin{lemma}\label{lem:hyperplanes}
Suppose $\Omega \subset \Cb^d$ is a convex domain and $H \subset \Cb^d$ is a complex hyperplane such that $H \cap \Omega = \emptyset$. Then for any $z_1, z_2 \in \Omega$ we have 
\begin{align*}
K_{\Omega}(z_1, z_2) \geq \frac{1}{2}\abs{ \log \frac{ d_{\Euc}(H,z_1)}{d_{\Euc}(H,z_2)} }.
\end{align*}
\end{lemma}

\begin{proof} Since $\Omega$ is convex, there exists a real hyperplane $H_{\Rb}$ so that $H \subset H_{\Rb}$ and $H_{\Rb} \cap \Omega = \emptyset$. By translating and rotating $\Omega$, we may assume that 
\begin{align*}
H_{\Rb} = \{ (z_1, \dots, z_d) \in \Cb^d : \Imaginary(z_1) =0\}
\end{align*}
and 
\begin{align*}
\Omega \subset \{ (z_1, \dots, z_d) \in \Cb^d : \Imaginary(z_1) >0\}.
\end{align*}
Consider the projection $P:\Cb^d \rightarrow \Cb$ given by $P(z_1, \dots, z_d) = z_1$. Then 
\begin{align*}
P(\Omega) \subset \Hc = \{ w \in \Cb : \Imaginary(w) > 0\}
\end{align*}
and so
\begin{align*}
K_{\Omega}(z_1,z_2) \geq K_{P(\Omega)}(P(z_1),P(z_2)) \geq K_{\Hc}(P(z_1), P(z_2)).
\end{align*}
Now for $w_1, w_2 \in \Hc$ we have 
\begin{align*}
K_{\Hc}(w_1, w_2)  
&= \frac{1}{2} \arcosh \left( 1+ \frac{ \abs{w_1-w_2}^2 }{2\Imaginary(w_1)\Imaginary(w_2)} \right)\\
 & \geq  \frac{1}{2} \arcosh \left( 1+ \frac{ (\abs{w_1}-\abs{w_2})^2  }{2\abs{w_1}\abs{w_2}} \right) =  \frac{1}{2} \arcosh \left( \frac{\abs{w_1}}{2\abs{w_2}} +\frac{\abs{w_2}}{2\abs{w_1}} \right) \\
 & = \frac{1}{2} \abs{\log \left( \frac{\abs{w_1}}{\abs{w_2}} \right) }.
 \end{align*}
So 
\begin{align*}
K_{\Omega}(z_1,z_2) \geq \frac{1}{2}\abs{ \log \frac{\abs{P(z_1)}}{\abs{P(z_2)}}}.
\end{align*}
Since $\abs{P(z_i)} = d_{\Euc}(H,z_i)$ this implies the lemma. 
\end{proof}

Suppose $\Omega$ is a domain with $C^1$ boundary. If $x \in \partial \Omega$ let $n_x$ be the inward pointing normal unit vector at $x$. 

\begin{proposition}\label{prop:rough_geodesics}
Suppose $\Omega \subset \Cb^d$ is a bounded convex domain with $C^{1,\alpha}$ boundary. Then there exists $\epsilon>0$ and $K \geq 1$ so that if $x \in \partial \Omega$ then the curve $\sigma_x: \Rb_{\geq 0} \rightarrow \Omega$ given by
\begin{align*}
\sigma_x(t)=x+\epsilon e^{-2t}n_x
\end{align*}
is an $K$-almost-geodesic. 
\end{proposition}

\begin{remark}The most difficult inequality to establish in the above proposition is the upper bound 
\begin{align*}
K_{\Omega}(\sigma_x(t), \sigma_x(s)) \leq \abs{t-s} + \log(K).
\end{align*}
To show this we will closely follow the proof of Proposition 2.5 in~\cite{FR1987}.\end{remark}

\begin{proof}
For $C, \rho > 0$ let 
\begin{align*}
\Dc_0 := \{ w \in \Cb : \abs{w} < \rho \text{ and } C\abs{\Imaginary(w)}^{1+\alpha} < \Real(w) \}.
\end{align*}
For $x \in \partial \Omega$ let $\phi_x: \Dc_0 \rightarrow \Cb^d$ be the map 
\begin{align*}
\phi_x(w) = x+ wn_x.
\end{align*}
Since $\partial \Omega$ is $C^{1,\alpha}$ we can pick $\rho, C >0$ so that $\phi_x(\Dc_0) \subset \Omega$ for all $x \in \partial \Omega$. 

Next let $\Dc \subset \Dc_0$ be a domain with $C^{1,\alpha}$ boundary, $0 \in \partial \Dc$, and symmetric about the real axis. Such a domain can be obtained by smoothing $\Dc_0$ near the two corner points. Now since $\Dc$ is symmetric about the real axis there exists a biholomorphic map $\varphi: \Dc \rightarrow \Delta$ with $\varphi(\Rb \cap \Dc) = \Rb \cap \Delta$ and 
\begin{align*}
\lim_{s \rightarrow 0} \varphi(s)=1.
\end{align*}
Since $\Dc$ has $C^{1,\alpha}$ boundary, $\varphi$ extends to a diffeomorphism $\overline{\Dc} \rightarrow \overline{\Delta}$ (see for instance~\cite[page 426 Theorem 6]{G1969}). 

Now fix $\epsilon >0$ and $\kappa \geq 1$ so that 
\begin{align*}
0 \leq 1- \kappa t \leq \varphi(t) \leq 1 - \frac{1}{\kappa}t
\end{align*}
for $t \in [0 , \epsilon]$.

Then for $0 < a < b \leq \epsilon$ we have
 \begin{align*}
 K_{\Dc_0}(a,b)
 & \leq K_{\Dc}(a,b) = K_\Delta( \varphi(a), \varphi(b)) = \frac{1}{2} \log \frac{(1 + \varphi(a))(1-\varphi(b))}{(1-\varphi(a))(1+\varphi(b))} \\
 & \leq \frac{1}{2} \log(2) + \frac{1}{2} \log \frac{ 1- \varphi(b)}{1- \varphi(a)} \\
 & \leq \frac{1}{2} \log(2) + \log(\kappa) + \frac{1}{2} \log (b/a).
 \end{align*}
 Thus if $\sigma_x(t) = x + \epsilon e^{-2t} n_x$ we have:
 \begin{align*}
K_{\Omega}(\sigma_x(t), \sigma_x(s)) \leq K_{\Dc_0}(e^{-2t}\epsilon, e^{-2s}\epsilon) \leq \frac{1}{2} \log(2) + \log(\kappa) + \abs{t-s}.
\end{align*}

On the other hand, 
\begin{align*}
K_{\Omega}(\sigma_x(t),\sigma_x(s)) \geq \frac{1}{2}\abs{ \log \frac{ d_{\Euc}(T_x^{\Cb} \partial \Omega ,\sigma_x(t))}{d_{\Euc}(T_x^{\Cb} \partial \Omega ,\sigma_x(s))} } = \abs{t-s}.
\end{align*}
Thus for any $x \in \partial \Omega$ the curve $\sigma_x$ is a $(1, \log \sqrt{2} \kappa )$-quasi-geodesic. 

Now since $\partial \Omega$ is $C^1$, by possibly decreasing $\epsilon > 0$ we can assume that 
\begin{align*}
\left\{ x + w n_x : \abs{w} \leq 2\epsilon \text{ and }\frac{1}{2}\abs{\Imaginary(w)} \leq \Real(w)\right\}  \subset \Omega
\end{align*}
for all $x \in \partial\Omega$. This implies that there exists a $C > 0$ so that 
\begin{align*}
\delta_{\Omega}(\sigma_x(t); \sigma_x^\prime(t)) \geq C\epsilon e^{-2t}
\end{align*}
for all $x \in \partial \Omega$ and $t \geq 0$. Then 
\begin{align*}
k_{\Omega}(\sigma_x(t); \sigma_x^\prime(t)) \leq \frac{\norm{\sigma_x^\prime(t)}}{\delta_{\Omega}(\sigma_x(t); \sigma_x^\prime(t))} = \frac{ 2\epsilon e^{-2t}}{\delta_{\Omega}(\sigma_x(t); \sigma_x^\prime(t))} \leq \frac{2}{C}
\end{align*}
and so 
\begin{align*}
K_{\Omega}(\sigma_x(t),\sigma_x(s)) \leq \frac{2}{C}\abs{t-s}.
\end{align*}
Thus $\sigma_x$ is a $(2/C, 0)$-quasi-geodesic. 

Thus for all $x \in \partial \Omega$ the curve $\sigma_x$ is an $K$-almost-geodesic with $K = \max\{ \sqrt{2} \kappa, 2/C\}$.
\end{proof}

\begin{lemma}\label{lem:dist} Suppose $\Omega$ is a bounded convex domain with $C^1$ boundary, $x, y \in \partial \Omega$, and $T_{x}^{\Cb} \partial \Omega \neq T_{y}^{\Cb} \partial \Omega$. Then there exists $\epsilon >0$ and $C \geq 0$ such that 
\begin{align*}
K_{\Omega}(p,q) \geq \frac{1}{2} \log \frac{1}{\delta_{\Omega}(p)} + \frac{1}{2} \log \frac{1}{\delta_{\Omega}(q)} - C
\end{align*}
when $p,q \in \Omega$, $d_{\Euc}(p, T_{x}^{\Cb} \partial \Omega) \leq \epsilon$, and $d_{\Euc}(q, T_{y}^{\Cb} \partial \Omega) \leq \epsilon$.
\end{lemma}

\begin{remark} Abate~\cite[Proposition 2.4.24, Corollary 2.4.25]{A1989} proved a weaker version of the above lemma assuming that $\Omega$ has $C^2$ boundary and $T_x \partial \Omega \neq T_y \partial \Omega$. 
\end{remark}

\begin{proof}
For a set $A \subset \Cb^d$ and $\delta \geq0$ let 
\begin{align*}
\Nc_{\delta}(A) = \{ z \in \Cb^d : d_{\Euc}(z,A) \leq \delta\}.
\end{align*}
For a point $p \in \partial \Omega$ let $\Pi(p):=\{p\}$ and for  a point $p \in \Omega$ let 
\begin{align*}
\Pi(p) : = \{ x \in \partial \Omega : \delta_{\Omega}(p) =d_{\Euc}(p,x)\}.
\end{align*}
Since $\partial \Omega$ is only $C^1$, we may have $\abs{\Pi(p)} > 1$ for $p \in \Omega$ arbitrarily close to $\partial \Omega$. 

Next for $\delta >0$ let
\begin{align*}
X(\delta) : =\overline{ \Omega} \cap \left( \cup \left\{ \Nc_{\delta}(T_{x_p}^{\Cb} \partial \Omega) : p \in \overline{\Omega} \cap \Nc_{\delta}(T_x^{\Cb} \partial \Omega) \text{ and } x_p \in \Pi(p) \right\}\right)
\end{align*}
and 
\begin{align*}
Y(\delta) : =\overline{\Omega} \cap \left( \cup \left\{ \Nc_{\delta}(T_{y_q}^{\Cb} \partial \Omega) : q \in \overline{\Omega} \cap \Nc_{\delta}(T_y^{\Cb} \partial \Omega) \text{ and } y_q \in \Pi(q) \right\}\right).
\end{align*}
Notice that $X(\delta)$ and $Y(\delta)$ are compact. 

We claim that there exists $\delta >0$ so that $X(\delta) \cap Y(\delta) = \emptyset$. Suppose not, then for each $n \in \Nb$ there exists $p_n \in \overline{ \Omega} \cap \Nc_{1/n}(T_x^{\Cb} \partial \Omega)$, $x_n \in \Pi(p_n)$, $q_n \in  \overline{\Omega} \cap \Nc_{1/n}(T_y^{\Cb} \partial \Omega)$, and $y_n \in \Pi(q_n)$ so that 
\begin{align*}
\overline{ \Omega} \cap \Nc_{1/n}(T_{x_n}^{\Cb} \partial \Omega) \cap \Nc_{1/n}(T_{y_n}^{\Cb} \partial \Omega) \neq \emptyset.
\end{align*}
Fix 
\begin{align*}
z_n \in \overline{ \Omega} \cap \Nc_{1/n}(T_{x_n}^{\Cb} \partial \Omega) \cap \Nc_{1/n}(T_{y_n}^{\Cb} \partial \Omega)
\end{align*}
and pass to a subsequence so that $x_n \rightarrow x^\prime$, $y_n \rightarrow y^\prime$, and $z_n \rightarrow z$. By construction $x^\prime \in T_x^{\Cb} \partial \Omega$ and $y^\prime \in T_{y}^{\Cb} \partial \Omega$. So $T_{x^\prime}^{\Cb} \partial \Omega = T_x^{\Cb} \partial \Omega$ and $T_{y^\prime}^{\Cb} \partial \Omega = T_y^{\Cb} \partial \Omega$. Thus 
\begin{align*}
z \in \overline{\Omega} \cap T_x^{\Cb} \partial \Omega \cap T_y^{\Cb} \partial \Omega
\end{align*}
which contradicts the fact that $\Omega$ is convex, $\partial \Omega$ is $C^1$, and $T_x^{\Cb} \partial \Omega \neq T_y^{\Cb} \partial \Omega$. So we can pick $\delta > 0$ so that $X(\delta) \cap Y(\delta) = \emptyset$.

Now since $X(\delta)$ and $Y(\delta)$ are compact there exists $r > 0$ so that 
\begin{align*}
\Nc_{r}(X(\delta)) \cap \Nc_{r}(Y(\delta)) = \emptyset.
\end{align*}

Now let $\epsilon = \min\{ \delta, r\}$. Suppose that $p \in \Omega \cap \Nc_{\epsilon}(T_x^{\Cb} \partial \Omega)$ and $q \in \Omega \cap \Nc_{\epsilon}(T_y^{\Cb} \partial \Omega)$. Then $p \in \Nc_{r}(X(\delta))$ and $q \in \Nc_{r}(Y(\delta))$. Moreover, if we pick $x_p \in \Pi(p)$ and $y_q \in \Pi(q)$ then 
\begin{align*}
 \Omega \cap \Nc_r(T^{\Cb}_{x_p}\partial \Omega) \subset X(\delta)
\end{align*}
 and 
 \begin{align*}
 \Omega \cap \Nc_r(T^{\Cb}_{y_q}\partial \Omega) \subset Y(\delta).
\end{align*}
Now let $\sigma:[0,T] \rightarrow \Omega$ be a geodesic with $\sigma(0)=p$ and $\sigma(T) = q$. Since 
\begin{align*}
\Nc_{r}(X(\delta)) \cap \Nc_{r}(Y(\delta)) = \emptyset
\end{align*}
there exists some $S \in [0,T]$ so that $\sigma(S) \notin \Nc_{r}(X(\delta)) \cup \Nc_{r}(Y(\delta))$. Then by Lemma~\ref{lem:hyperplanes} we have
\begin{align*}
K_\Omega(p,q) 
& = K_\Omega(p,\sigma(S)) + K_\Omega(\sigma(S), q) \\
& \geq \frac{1}{2} \log \frac{d_{\Euc}(\sigma(S), T^{\Cb}_{x_p}\partial \Omega)}{d_{\Euc}(p, T^{\Cb}_{x_p}\partial \Omega)} + \frac{1}{2} \log \frac{d_{\Euc}(\sigma(S), T^{\Cb}_{y_q}\partial \Omega)}{d_{\Euc}(q, T^{\Cb}_{y_q}\partial \Omega)}  \\
& \geq  \frac{1}{2} \log \frac{1}{\delta_{\Omega}(p)} + \frac{1}{2} \log \frac{1}{\delta_{\Omega}(q)} + \log(r)
\end{align*}
\end{proof}
 
\begin{proof}[Proof of Theorem~\ref{thm:gromov_prod}]
Pick $\epsilon >0$ and $K \geq1$ so that the curve $\sigma_z: \Rb_{\geq 0} \rightarrow \Omega$ given by
\begin{align*}
\sigma_z(t)=z+e^{-2t}\epsilon n_z
\end{align*}
is an $K$-almost-geodesic for any $z \in \partial \Omega$. By compactness, there exists $R \geq 0$ so that 
\begin{align*}
K_{\Omega}(\sigma_z(0), o) \leq R
\end{align*}
for all $z \in \partial \Omega$. 

First suppose that $p_n, q_m \rightarrow x \in \partial \Omega$. Next let $x_n$ be a closest point in $\partial \Omega$ to $p_n$ and $y_m$ be a closest point in $\partial \Omega$ to $q_m$. Then we can suppose that $p_n=\sigma_{x_n}(s_n)$ and $q_m=\sigma_{y_m}(t_m)$ for some $s_n,t_m \rightarrow \infty$. 

Now fix $T > 0$. Then for $n,m$ large enough we have $s_n, t_m \geq T$ and 
\begin{align*}
2(p_n|q_m)_o = K_{\Omega}(p_n,o)+K_{\Omega}(q_m,o) -K_{\Omega}(p_n, q_m) \geq s_n+t_m-2R-2K- K_{\Omega}(p_n,q_m).
\end{align*}
Now 
\begin{align*}
K_{\Omega}(p_n,q_m) &\leq K_{\Omega}(\sigma_{x_n}(T), \sigma_{y_m}(T)) + K_{\Omega}(\sigma_{x_n}(T), p_n) + K_{\Omega}(\sigma_{y_m}(T), q_m) \\
&\leq K_{\Omega}(\sigma_{x_n}(T), \sigma_{y_m}(T)) + (s_n-T) + (t_m-T)+2K
\end{align*}
and so 
\begin{align*}
(p_n|q_m)_o\geq T - R - \frac{1}{2}K_{\Omega}(\sigma_{x_n}(T), \sigma_{y_m}(T))-2K.
\end{align*}
Since $x_n, y_m \rightarrow x$ we see that  $K_{\Omega}(\sigma_{x_n}(T), \sigma_{y_m}(T)) \rightarrow 0$ and so 
\begin{align*}
\liminf_{n,m \rightarrow \infty} (p_n|q_m)_o \geq T - R-2K.
\end{align*}
Since $T > 0$ was arbitrary this implies part (1) of the theorem. 

We now prove part (2). Suppose for a contradiction that 
\begin{align*}
\limsup_{n,m \rightarrow \infty} ( p_n | q_m)_o = \infty
\end{align*}
and $T_{x}^{\Cb} \partial \Omega \neq T_{y}^{\Cb} \partial \Omega$. Now by Proposition~\ref{prop:rough_geodesics} and Lemma~\ref{lem:dist} there exists $K > 0$ so that 
\begin{align*}
K_{\Omega}(o,p_n) \leq K + \frac{1}{2} \log \frac{1}{\delta_{\Omega}(p_n)},
\end{align*}
\begin{align*}
K_{\Omega}(o,q_m) \leq K + \frac{1}{2} \log \frac{1}{\delta_{\Omega}(q_m)},
\end{align*}
and 
\begin{align*}
K_{\Omega}(p_n,q_m) \geq \frac{1}{2} \log \frac{1}{\delta_{\Omega}(p_n)} + \frac{1}{2} \log \frac{1}{\delta_{\Omega}(q_m)}  - K
\end{align*}
for all $n,m \geq 0$. Thus 
\begin{align*}
(p_n|q_m)_o \leq \frac{3}{2}K.
\end{align*}
So we have contradiction and thus $T_{x}^{\Cb} \partial \Omega = T_{y}^{\Cb} \partial \Omega$
\end{proof}

\section{A Wolff-Denjoy theorem}

In this section we use Theorem~\ref{thm:gromov_prod} to prove Theorem~\ref{thm:wolf_i} which we restate: 

\begin{theorem}\label{thm:wolf}
 Suppose $\Omega \subset \Cb^d$ is a bounded convex domain with $C^{1,\alpha}$ boundary. If $f:\Omega \rightarrow \Omega$ is 1-Lipschitz with respect to the Kobayashi metric, then $f$ either has a fixed point in $\Omega$ or there exists some $x \in \partial \Omega$ so that 
 \begin{align*}
 \lim_{k \rightarrow \infty} d_{Euc}(f^{k}(p) , T_x^{\Cb} \partial \Omega) = 0
 \end{align*}
 for all $p \in \Omega$. 
 \end{theorem}
 
The proof of Theorem~\ref{thm:wolf} uses Theorem~\ref{thm:gromov_prod} and a result of Karlsson about the iterations of 1-Lipschitz maps on general metric spaces. In particular, the argument below is essentially the proof of Theorem 3.4 in~\cite{K2001} adapted to this specific setting. 
 
 \begin{proof}
 For $p \in \Omega$ and $R \geq 0$ let $B_{\Omega}(p;R)$ be the closed ball of radius $R$ centered at $p$ with respect to the Kobayashi metric. By Proposition 2.3.46 in~\cite{A1989} $B_{\Omega}(p;R)$ is a closed convex subset of $\Cb^d$. 
 
 Fix $o \in \Omega$. Then by~\cite[Theorem 5.6]{C1984b} either 
 \begin{align*}
 \sup_{n \geq 0} K_{\Omega}(f^n (o), o) < \infty
 \end{align*}
 or 
 \begin{align*}
 \lim_{n \rightarrow \infty} K_{\Omega}(f^n (o), o) = \infty.
 \end{align*}
 
First suppose that 
 \begin{align*}
 \sup_{n \geq 0} K_{\Omega}(f^n (o), o) < \infty.
 \end{align*}
 We claim that $f$ has a fixed point in $\Omega$. Let 
 \begin{align*}
 \Cc = \{ K \subset \Omega : K \text{ compact and } f(K) \subset K\}.
 \end{align*}
Notice that 
 \begin{align*}
 \overline{\{f^n(o) : n \geq 0\}} \in \Cc
 \end{align*}
 so $\Cc$ is non-empty. Then by Zorn's Lemma there exists a minimal element $K_0 \in \Cc$. Since $f(K_0)$ is also in $\Cc$ we see that $f(K_0)=K_0$. Now pick $R >0$ so that 
 \begin{align*}
 C = \cap_{k \in K_0} B_{\Omega}(k; R)
 \end{align*}
 is non-empty. Then $C$ is closed, convex, and $f(C) \subset C$. Thus by Brouwer's fixed-point theorem there exists some $c \in C$ so that $f(c)=c$. 
 
 Next suppose that
  \begin{align*}
 \lim_{n \rightarrow \infty} K_{\Omega}(f^n (o), o) = \infty.
 \end{align*}
Then pick a subsequence $n_i  \rightarrow \infty$ so that 
\begin{align*}
K_{\Omega}(f^{n_i} (o), o) > K_{\Omega}(f^{m} (o), o)
\end{align*}
for all $m < n_i$. By passing to another subsequence we may suppose that $f^{n_i} (o) \rightarrow x \in \partial \Omega$. 

Suppose that $p \in \Omega$ and $f^{m_j} (p) \rightarrow x^\prime$ for some sequence $m_j\rightarrow \infty$. We claim that $x^\prime \in T_{x}^{\Cb} \partial \Omega$. Pick a sequence $i_j \rightarrow \infty$ with $n_{i_j} > m_j$ then 
\begin{align*}
2( f^{n_{i_j}} (o) | f^{m_j} (p) )_o 
&= K_{\Omega}(  f^{n_{i_j}} (o), o) + K_{\Omega}(o, f^{m_j} (p)) - K_{\Omega}(f^{n_{i_j}} (o), f^{m_j} (p) ) \\
&\geq K_{\Omega}(  f^{n_{i_j}} (o), o) + K_{\Omega}(o, f^{m_j} (p)) - K_{\Omega}(f^{n_{i_j}-m_j} (o), o )-K_{\Omega}(p,o) \\
& \geq K_{\Omega}(o, f^{m_j} (p))-K_{\Omega}(p,o).
\end{align*}
So 
\begin{align*}
\lim_{j \rightarrow \infty} ( f^{n_{i_j}} (o) | f^{m_j} (p) )_o  = \infty
\end{align*}
and hence $x^\prime \in T_{x}^{\Cb} \partial \Omega$ by Theorem~\ref{thm:gromov_prod}.
\end{proof}

\section{The behavior of geodesics}

In this section we use Theorem~\ref{thm:gromov_prod} to prove a version of Theorem~\ref{thm:gromov_visible} for the Kobayashi metric on convex domains. 

\begin{theorem}\label{thm:geod_converge}
Suppose $\Omega \subset \Cb^d$ is a bounded convex domain with $C^{1,\alpha}$ boundary and $p_n, q_n \in \Omega$ are sequences such that $p_n \rightarrow x \in \partial \Omega$ and $q_n \rightarrow y \in \partial \Omega$ with $T_{x}^{\Cb} \partial \Omega \neq T_{y} \partial \Omega$. 

If $\sigma_n:[0,T_n] \rightarrow \Omega$ is an $K$-almost-geodesic with $\sigma_n(0)=q_n$ and $\sigma_n(T_n) = p_n$, then there exists $n_k \rightarrow \infty$ and $S_k \in [0,T_{n_k}]$ so that the $K$-almost-geodesics 
\begin{align*}
t \rightarrow \sigma_{n_k}(t+S_k)
\end{align*}
converge locally uniformly to an $K$-almost-geodesic $\sigma:\Rb \rightarrow\Omega$. Moreover, 
\begin{align*}
\lim_{t \rightarrow \infty} d_{\Euc}(\sigma(t), T_{x}^{\Cb} \partial \Omega) = 0
\end{align*}
and 
\begin{align*}
\lim_{t \rightarrow -\infty} d_{\Euc}(\sigma(t), T_{y}^{\Cb} \partial \Omega) = 0.
\end{align*}
\end{theorem}

\begin{remark} Notice that Theorem~\ref{thm:geod_converge} implies Theorem~\ref{thm:visible} from Section~\ref{sec:outline}. \end{remark}

\begin{proof}
We first claim that if $n_k \rightarrow \infty$, $S_k \in [0,T_k]$, and $\sigma_{n_k}(S_k) \rightarrow z \in \partial \Omega$ then 
\begin{align*}
z \in T_{x}^{\Cb} \partial \Omega \cup T_{y}^{\Cb} \partial \Omega.
\end{align*}
Suppose not then $T_z^{\Cb} \partial \Omega \neq T_x^{\Cb} \partial \Omega$ and $T_z^{\Cb} \partial \Omega \neq T_y^{\Cb} \partial \Omega$.  So by Lemma~\ref{lem:dist} there exists $K_1 \geq 0$ so that 
\begin{align*}
K_{\Omega}(p_{n_k}, q_{n_k}) 
&\geq K_{\Omega}(p_{n_k}, \sigma_{n_k}(S_k) ) + K_\Omega(\sigma_{n_k}(S_k) , q_{n_k}) -3K \\
& \geq \frac{1}{2} \log \frac{1}{\delta_\Omega(p_{n_k})} +  \log \frac{1}{\delta_\Omega(\sigma_{n_k}(S_k))} +  \frac{1}{2} \log \frac{1}{\delta_\Omega(q_{n_k})} - 3K-2K_1.
\end{align*}
However by Proposition~\ref{prop:rough_geodesics} there exists $K_2 \geq 0$ so that 
\begin{align*}
K_{\Omega}(p_{n_k}, q_{n_k}) \leq  \frac{1}{2} \log \frac{1}{\delta_\Omega(p_{n_k})} +  \frac{1}{2} \log \frac{1}{\delta_\Omega(q_{n_k})} + K_2.
\end{align*}
Combining the two inequalities implies that 
\begin{align*}
\log \frac{1}{\delta_\Omega(\sigma_{n_k}(S_k))}  \leq 3K+2K_1 + K_2
\end{align*}
which is impossible since $\delta_\Omega(\sigma_{n_k}(S_k)) \rightarrow 0$.

Now there exists $\epsilon > 0$ so that the sets 
\begin{align*}
U = \{ z \in \Omega : d_{\Euc}(z, T_{x}^{\Cb} \partial \Omega) \leq \epsilon \}
\end{align*}
and
\begin{align*}
V = \{ z \in \Omega: d_{\Euc}(z, T_{y}^{\Cb} \partial \Omega) \leq \epsilon \}
\end{align*}
 are disjoint. Pick $S_n\in [0,T_n]$ so that $\sigma_n(S_n) \in \Omega \setminus (U \cup V)$ and pass to a subsequence so that $\sigma_n(S_n) \rightarrow z \in \overline{\Omega}$. Now by construction and the claim above we must have that $z \in \Omega$. Then since each $\sigma_n$ is $K$-Lipschitz (with respect to the Kobayashi metric) we can pass to a subsequence so that the $K$-almost-geodesics 
\begin{align*}
t \rightarrow \sigma_n(t+S_n)
\end{align*}
converge locally uniformly to an $K$-almost-geodesic $\sigma:\Rb \rightarrow\Omega$. 

Next we claim that 
\begin{align*}
\lim_{t \rightarrow \infty} d_{\Euc}(\sigma(t), T_{x}^{\Cb} \partial \Omega) = 0.
\end{align*}
Fix some sequence $s_m \rightarrow \infty$ so that $\sigma(s_m)$ converges to some $x^\prime \in \partial \Omega$. Now
\begin{align*}
\lim_{m,t \rightarrow \infty} (\sigma(s_m)|\sigma(t))_{\sigma(0)} \geq \frac{1}{2} \left(\lim_{n,t \rightarrow \infty} \min\{s_m,t\} - 3K \right)= \infty
\end{align*}
and so by Theorem~\ref{thm:gromov_prod}
\begin{align*}
\lim_{t \rightarrow \infty} d_{\Euc}(\sigma(t), T_{x^\prime}^{\Cb} \partial \Omega) = 0.
\end{align*}
On the other hand, since $\sigma_n(\cdot + S_n)$ converges locally uniformly to $\sigma$ we can pick $s_n^\prime \rightarrow \infty$ so that $\sigma_n(s_n^\prime+S_n) \rightarrow x^\prime$. Then 
\begin{align*}
\lim_{n \rightarrow \infty} (\sigma_n(s_n^\prime+S_n)|p_n)_{\sigma(0)}  
&\geq \lim_{n \rightarrow \infty} (\sigma_n(s_n^\prime+S_n)|\sigma_n(T_n))_{\sigma_n(0)}-K_{\Omega}(\sigma_n(0), \sigma(0)) \\
& \geq \frac{1}{2} \left( \lim_{n \rightarrow \infty} \min\{s_n^\prime+S_n,T_n\} - 3K \right) = \infty.
\end{align*}
So $x^\prime \in T_{x}^{\Cb} \partial \Omega$ by Theorem~\ref{thm:gromov_prod} and thus $T_{x^\prime}^{\Cb} \partial \Omega = T_{x}^{\Cb} \partial \Omega$.

The proof that 
\begin{align*}
\lim_{t \rightarrow -\infty} d_{\Euc}(\sigma(t), T_{y}^{\Cb} \partial \Omega) = 0
\end{align*}
is identical.
\end{proof}

\section{Finding a hyperbolic element}\label{sec:finding_hyp}

The main purpose of this section is to prove the following existence result for hyperbolic elements: 

 \begin{theorem}\label{thm:exist_hyp}
  Suppose $\Omega \subset \Cb^d$ is a bounded convex domain with $C^{1,\alpha}$ boundary. Then $\Aut(\Omega)$ contains a hyperbolic element if and only if there exists $x, y \in \Lc(\Omega)$ with $T_{x}^{\Cb} \partial \Omega \neq T_{y}^{\Cb} \partial \Omega$.
   \end{theorem}

The proof of Theorem~\ref{thm:exist_hyp} will require a number of preliminary results concerning the behavior of elliptic, hyperbolic, and parabolic elements. 

\subsection{Invariance of accumulation set} 
 
\begin{definition} Suppose $\Omega \subset \Cb^d$ is a bounded convex domain with $C^{1,\alpha}$ boundary. For $\varphi \in \Aut(\Omega)$ non-elliptic let $\Lc(\Omega, \varphi)$ be the set of points $x \in \partial \Omega$ where there exists some $p \in \Omega$ and some sequence $m_i \rightarrow \infty$ such that $\varphi^{m_i} p \rightarrow x$. 
\end{definition}

Notice that 
\begin{align*}
\Lc(\Omega, \varphi) \subset \Lc(\Omega) \cap H_{\varphi}^+
\end{align*}
by Theorem~\ref{thm:wolf}. Moreover the subset $\Lc(\Omega, \varphi) \subset \partial \Omega$ is invariant by $\varphi$ in the following sense:

\begin{lemma}\label{lem:point_attract}
Suppose $\Omega \subset \Cb^d$ is a bounded convex domain with $C^{1,\alpha}$ boundary and $\varphi \in \Aut(\Omega)$ is non-elliptic. If $x \in \Lc(\Omega,\varphi)$ and $p_n \in \Omega$ converges to $x$, then 
\begin{align*}
\lim_{n \rightarrow \infty} d_{\Euc}(\varphi^{k} p_n, H_{\varphi}^+) =0
\end{align*}
for any $k \in \Zb$.
\end{lemma}

\begin{proof}
Suppose that  $\varphi^{m_i} p \rightarrow x \in \partial \Omega$ and $p_n \rightarrow x$. Then by Theorem~\ref{thm:gromov_prod}
\begin{align*}
\lim_{n, i \rightarrow \infty} ( p_n | \varphi^{m_i} p)_p = \infty.
\end{align*}
Moreover 
\begin{align*}
( \varphi^k p_n | \varphi^{m_i} p)_p 
& =  ( p_n | \varphi^{m_i-k} p)_{\varphi^{-k}p} \geq  ( p_n | \varphi^{m_i} p)_{p} - K_{\Omega}(\varphi^{m_i-k} p , \varphi^{m_i} p)-K_{\Omega}(\varphi^{-k}p , p) \\
&= ( p_n | \varphi^{m_i} p)_{p} -2K_{\Omega}(\varphi^{-k}p , p).
\end{align*}
So 
\begin{align*}
\lim_{n,i \rightarrow \infty} ( \varphi^k p_n | \varphi^{m_i} p)_p  = \infty
\end{align*}
and then by Theorem~\ref{thm:gromov_prod}
\begin{align*}
\lim_{n \rightarrow \infty} d_{\Euc}(\varphi^{k} p_n, H_{\varphi}^+) =0.
\end{align*}
\end{proof}

\subsection{Continuity of attracting hyperplanes} 

\begin{lemma}\label{lem:cont_att}
Suppose $\Omega \subset \Cb^d$ is a bounded convex domain with $C^{1,\alpha}$ boundary, $p \in \Omega$, and $\varphi_n \in \Aut(\Omega)$. If each $\varphi_n$ is not elliptic and $\varphi_n p \rightarrow x \in \partial \Omega$,  then $H_{\varphi_n}^+ \rightarrow T_{x}^{\Cb} \partial \Omega$ in the space of complex hyperplanes. 
\end{lemma}

\begin{proof}
By compactness, it is enough to show that every convergent subsequence of $H^+_{\varphi_n}$ converges to $T_{x}^{\Cb} \partial \Omega$. So without loss of generality, assume that $H^+_{\varphi_n}$ converges to some complex hyperplane $H$. Now for each $n \in \Nb$
\begin{align*}
\lim_{m \rightarrow \infty} d_{\Euc}(\varphi_n^m p, H_{\varphi_n}^+) = 0.
\end{align*}
Select $m_n$ so that  
\begin{align*}
d_{\Euc}(\varphi_n^{m_n} p, H_{\varphi_n}^+) \leq 1/n
\end{align*}
and
\begin{align*}
K_{\Omega}(\varphi_n^{m_n} p, p ) \geq K_{\Omega}(\varphi_n^{m_n-1} p, p ).
\end{align*}
By passing to another subsequence we can suppose that $\varphi_n^{m_n}p \rightarrow x^\prime \in \partial \Omega$. By construction $x^\prime \in H$ and so $T_{x^\prime}^{\Cb} \partial \Omega = H$. 

However, 
\begin{align*}
2( \varphi^{m_n}_n p | \varphi_n p)_p 
&= K_{\Omega}( \varphi^{m_n}_n p, p) + K_\Omega(p, \varphi_n p)- K_{\Omega}( \varphi^{m_n}_n p,  \varphi_n p) \\
& = K_{\Omega}( \varphi^{m_n}_n p, p) + K_\Omega(p, \varphi_n p)- K_{\Omega}( \varphi^{m_n-1}_n p,  p) \\
& \geq K_\Omega(p, \varphi_n p)
\end{align*}
and so 
\begin{align*}
\lim_{n \rightarrow \infty} ( \varphi^{m_n}_n p | \varphi_n p)_p  = \infty.
\end{align*}
Thus by Theorem~\ref{thm:gromov_prod} $x^\prime \in T_{x}^{\Cb} \partial \Omega$ which implies that 
\begin{align*}
H = T_{x^\prime}^{\Cb} \partial \Omega = T_{x}^{\Cb} \partial \Omega.
\end{align*}
\end{proof}

\subsection{Parabolic automorphisms}

\begin{lemma}\label{lem:para_inverse}
Suppose $\Omega \subset \Cb^d$ is a bounded convex domain with $C^{1,\alpha}$ boundary, $\varphi_n \in \Aut(\Omega)$, and $p \in \Omega$. If $\varphi_n p \rightarrow x \in \partial \Omega$ and each $\varphi_n$ is parabolic, then
\begin{align*}
\lim_{n \rightarrow \infty} d_{\Euc}(\varphi_n q, T_{x}^{\Cb} \partial \Omega) =0 =  \lim_{n \rightarrow \infty} d_{\Euc}(\varphi_n^{-1} q, T_{x}^{\Cb} \partial \Omega)
\end{align*}
for all $q \in \Omega$. 
\end{lemma}

\begin{proof} This follows immediately from Lemma~\ref{lem:cont_att}. \end{proof}

\subsection{Elliptic automorphisms}

\begin{lemma}\label{lem:elliptic_fixed_pt}
Suppose $\Omega \subset \Cb^d$ is a bounded convex domain with $C^{1,\alpha}$ boundary, $\varphi_n \in \Aut(\Omega)$, and $p \in \Omega$. If $\varphi_n p \rightarrow x \in \partial \Omega$, each $\varphi_n$ is elliptic, and $e_n \in \Omega$ is a fixed point of $\varphi_n$, then 
\begin{align*}
\lim_{n \rightarrow \infty} d_{\Euc}(e_n, T_{x}^{\Cb} \partial \Omega) = 0.
\end{align*}
\end{lemma}

\begin{proof} Fix $o \in \Omega$ then
\begin{align*}
2(\varphi_n p|e_n)_o 
&= K_\Omega(\varphi_n p,o)+K_\Omega(e_n,o) - K_{\Omega}(\varphi_n p, e_n) \\
& = K_\Omega(\varphi_n p,o)+K_\Omega(e_n,o) - K_{\Omega}(p, e_n) \\
& \geq K_\Omega(\varphi_n p,o)- K_\Omega(p,o).
\end{align*}
So 
\begin{align*}
\lim_{n \rightarrow \infty} (\varphi_n p|e_n)_o = \infty
\end{align*}
and then Theorem~\ref{thm:gromov_prod} implies the lemma.
\end{proof}

\begin{lemma}\label{lem:elliptic_inverse}
Suppose $\Omega \subset \Cb^d$ is a bounded convex domain with $C^{1,\alpha}$ boundary, $\varphi_n \in \Aut(\Omega)$, and $p \in \Omega$. If $\varphi_n p \rightarrow x \in \partial \Omega$ and each $\varphi_n$ is elliptic, then
\begin{align*}
\lim_{n \rightarrow \infty} d_{\Euc}(\varphi_n q, T_{x}^{\Cb} \partial \Omega) =0 = \lim_{n \rightarrow \infty} d_{\Euc}(\varphi_n^{-1} q, T_{x}^{\Cb} \partial \Omega)
\end{align*}
for all $q \in \Omega$. 
\end{lemma}

\begin{proof}
Suppose for a contradiction that there exists $q \in \Omega$, $n_k \rightarrow \infty$, and $\delta_k \in \{-1,1\}$ so that
\begin{align*}
\lim_{k \rightarrow} d_{\Euc}(\varphi_{n_k}^{\delta_k} q, T_{x}^{\Cb} \partial \Omega)  \neq 0.
\end{align*}
After passing to a subsequence we can assume that $\varphi_{n_k}^{\delta_k} q \rightarrow x^\prime \in \partial \Omega$. Then we must have $T_x^{\Cb} \partial \Omega \neq T_{x^\prime}^{\Cb} \partial \Omega$. 

Now if $e_k$ is a fixed point of $\varphi_{n_k}$ we see from Lemma~\ref{lem:elliptic_fixed_pt} that 
\begin{align*}
\lim_{k \rightarrow \infty} d_{\Euc}(e_k, T_{x^\prime}^{\Cb} \partial \Omega) = 0.
\end{align*}
On the other hand, since $\varphi_{n_k} p \rightarrow x$, Lemma~\ref{lem:elliptic_fixed_pt} also implies that 
\begin{align*}
\lim_{k \rightarrow \infty} d_{\Euc}(e_k, T_{x}^{\Cb} \partial \Omega) = 0.
\end{align*}
So 
\begin{align*}
T_x^{\Cb} \partial \Omega \cap T_{x^\prime}^{\Cb} \partial \Omega \cap \partial \Omega \neq \emptyset
\end{align*}
and so $T_x^{\Cb} \partial \Omega = T_{x^\prime}^{\Cb} \partial \Omega$. Thus we have a contradiction.

\end{proof}

\subsection{Uniform attraction}

 \begin{proposition}\label{prop:uniform_att} Suppose $\Omega \subset \Cb^d$ is a bounded convex domain with $C^{1,\alpha}$ boundary, $\varphi_n \in \Aut(\Omega)$, $o \in \Omega$, $\varphi_n o \rightarrow x \in \partial \Omega$, and each $\varphi_n$ is either elliptic or parabolic. If $U \subset \Cb^d$ is a neighborhood of $T_{x}^{\Cb} \partial \Omega \cap \partial \Omega$ then there exists $N \geq 0$ so that 
\begin{align*}
\varphi_n(\Omega \setminus U) \subset U
\end{align*}
for all $n \geq N$.
\end{proposition}

\begin{proof}
Since each $\varphi_n$ is either elliptic or parabolic, Lemma~\ref{lem:para_inverse} and Lemma~\ref{lem:elliptic_inverse} imply that
\begin{align}
\label{eq:attractive}
\lim_{n \rightarrow \infty} d_{\Euc}(\varphi_n q, T_{x}^{\Cb} \partial \Omega) = 0 = \lim_{n \rightarrow \infty} d_{\Euc}(\varphi_n^{-1} q, T_{x}^{\Cb} \partial \Omega)
\end{align}
for every $q \in \Omega$. 

Now suppose for a contradiction that the proposition does not hold. Then after passing to a subsequence there exists $p_n \in \Omega \setminus U$ so that $\varphi_n p_n \notin U$. By passing to another subsequence we can suppose that $p_n \rightarrow y_1\in  \overline{\Omega} \setminus U$ and $\varphi_{n} p_n \rightarrow y_2 \in \overline{\Omega} \setminus U$. We will use Theorem~\ref{thm:geod_converge} to show that this is impossible. 

We first claim that $y_1 \in \partial \Omega$. If not then Equation~\ref{eq:attractive} implies that
\begin{align*}
\lim_{n \rightarrow \infty} d_{\Euc}(\varphi_{n} y_1, T_{x}^{\Cb} \partial \Omega) = 0.
\end{align*}
Since 
\begin{align*}
\lim_{n \rightarrow \infty} K_{\Omega}(\varphi_{n} p_n, \varphi_{n} y_1)=\lim_{n \rightarrow \infty} K_{\Omega}(p_n, y_1) = 0
\end{align*}
we see that
\begin{align*}
\lim_{k \rightarrow \infty} d_{\Euc}(\varphi_{n} p_n, \varphi_{n} y_1)=0
\end{align*}
and so
\begin{align*}
\lim_{k \rightarrow \infty} d_{\Euc}(\varphi_{n} p_n, T_{x}^{\Cb} \partial \Omega) = 0.
\end{align*}
But this contradicts the fact that $\varphi_{n} p_n \notin U$. So $y_1 \in \partial \Omega$. A similar argument shows that $y_2 \in \partial \Omega$. Since $U$ is a neighborhood of $T_x^{\Cb} \partial \Omega \cap \partial \Omega$ we then see that $T_{y_1}^{\Cb} \partial \Omega \neq T_x^{\Cb} \partial \Omega$ and $T_{y_2}^{\Cb} \partial \Omega \neq T_x^{\Cb} \partial \Omega$.

We now claim that after possibly passing to a subsequence of the $\varphi_n$ we can find a sequence $q_n \in \Omega$ so that $q_n \rightarrow x$ and 
\begin{align*}
\varphi_{n} q_n \rightarrow x^\prime \in \partial \Omega \cap T_{x}^{\Cb} \partial \Omega.
\end{align*} 
To see this fix any sequence $q_m \in \Omega$ with $q_m \rightarrow x$. With $m$ fixed we have
\begin{align*}
\liminf_{n \rightarrow \infty} ( \varphi_{n} q_m | \varphi_{n} o)_{o} 
&\geq \liminf_{n \rightarrow \infty} \Big( ( \varphi_{n} o | \varphi_{n} o)_{o} - K_{\Omega}(\varphi_{n} q_m, \varphi_{n} o) \Big)\\
&=  \liminf_{n \rightarrow \infty} \Big( ( \varphi_{n} o | \varphi_{n} o)_{o} - K_{\Omega}(q_m, o)\Big)=\infty.
\end{align*}
So we can find $n_m \rightarrow \infty$ so that
\begin{align*}
\lim_{k \rightarrow \infty} ( \varphi_{n_m} q_m | \varphi_{n_m} o)_{o}=\infty.
\end{align*}
So by Theorem~\ref{thm:gromov_prod} and passing to another subsequence we can suppose that 
\begin{align*}
\varphi_{n_m} q_m \rightarrow x^\prime \in \partial \Omega \cap T_{x}^{\Cb} \partial \Omega.
\end{align*} 
Now replace $\varphi_n$ with the subsequence $\varphi_{n_m}$. 

Now for each $n$, let $\sigma_n :[a_n, b_n] \rightarrow \Omega$ be a geodesic with $\sigma_n(a_n)=p_n$ and $\sigma_n(b_n) = q_n$. Then since 
\begin{align*}
T_{y_1}^{\Cb} \partial \Omega \neq  T_{x}^{\Cb} \partial \Omega,
\end{align*}
using Theorem~\ref{thm:geod_converge} we can parametrize $\sigma_n$ and pass to a subsequence so that $\sigma_n$ converges locally uniformly to a geodesic $\gamma_1: \Rb \rightarrow \Omega$. 

Since $\lim_{n \rightarrow \infty} \varphi_{n} \sigma_n(a_n) = y_2$, $\lim_{k \rightarrow \infty} \varphi_{n} \sigma_n(b_n) = x^\prime$, and
\begin{align*}
T_{y_2}^{\Cb} \partial \Omega \neq  T_x^{\Cb} \partial \Omega = T_{x^\prime}^{\Cb} \partial \Omega
\end{align*}
by passing to another subsequence we can find $S_n \in [a_n, b_n]$ so that the geodesics $t \rightarrow \varphi_{n} \sigma_n(t+S_n)$ converge locally uniformly to a geodesic $\gamma_2: \Rb \rightarrow \Omega$. 

Now by Equation~\ref{eq:attractive}
\begin{align*}
\lim_{n \rightarrow \infty} \abs{S_n} = \infty.
\end{align*}
So after passing to a subsequence we have two cases: \newline

\noindent \textbf{Case 1:} $\lim_{n \rightarrow \infty} S_n = \infty$. Then 
\begin{align*}
(\varphi_{n} \gamma_1(0) | \varphi_n p_n)_{\gamma_2(0)}
& \geq (\varphi_n \sigma_n(0) | \varphi_n p_n)_{\varphi_n \sigma_n(S_n)} \\
& \quad - K_{\Omega}( \varphi_{n} \gamma_1(0), \varphi_n \sigma_n(0)) - K_{\Omega}(\gamma_2(0), \varphi_n \sigma_n(S_n)).
\end{align*}
Since $\varphi_n \sigma_n$ is a geodesic, $\varphi_n \sigma_n(a_n) = \varphi_n p_n$, and $a_n \leq 0 \leq S_n$ we have
\begin{align*}
(\varphi_n \sigma_n(0) | \varphi_n p_n)_{\varphi_n \sigma_n(S_n)} = S_n
\end{align*}
and so
\begin{align*}
(\varphi_{n} \gamma_1(0) | \varphi_n p_n)_{\gamma_2(0)} \geq S_n - K_{\Omega}(\gamma_1(0), \sigma_n(0)) - K_{\Omega}(\gamma_2(0), \varphi_n \sigma_n(S_n)).
\end{align*}
Thus 
\begin{align*}
\lim_{n \rightarrow \infty} (\varphi_{n} \gamma_1(0) | \varphi_n p_n)_{\gamma_2(0)} = \infty
\end{align*}
and then by Theorem~\ref{thm:gromov_prod} we have
\begin{align*}
\lim_{n \rightarrow \infty} d_{\Euc}(\varphi_{n} \gamma_1(0), T_{y_2}^{\Cb} \partial \Omega) = 0.
\end{align*}
But this contradicts Equation~\ref{eq:attractive}. \newline

\noindent \textbf{Case 2:} $\lim_{n \rightarrow \infty} S_n = -\infty$. Then 
\begin{align*}
(\varphi_{n}^{-1} \gamma_2(0) | p_n)_{\gamma_1(0)}
& \geq (\sigma_n(S_n) | p_n )_{\sigma_n(0)} - K_{\Omega}( \varphi_{n}^{-1} \gamma_2(0), \sigma_n(S_n)) - K_{\Omega}(\gamma_1(0), \sigma_n(0)). 
\end{align*}
Since $\sigma_n$ is a geodesic, $\sigma_n(a_n)=p_n$, and $a_n \leq S_n \leq 0$ we have
\begin{align*}
(\sigma_n(S_n) | p_n )_{\sigma_n(0)} = -S_n
\end{align*}
and so 
\begin{align*}
(\varphi_{n}^{-1} \gamma_2(0) | p_n)_{\gamma_1(0)} \geq -S_n - K_{\Omega}(\gamma_2(0), \varphi_n\sigma_n(S_n)) - K_{\Omega}(\gamma_1(0), \sigma_n(0)).
\end{align*}
Thus 
\begin{align*}
\lim_{n \rightarrow \infty} (\varphi_{n}^{-1} \gamma_2(0) | p_n)_{\gamma_1(0)} = \infty
\end{align*}
and then by Theorem~\ref{thm:gromov_prod} we have
\begin{align*}
\lim_{n \rightarrow \infty} d_{\Euc}(\varphi_{n}^{-1} \gamma_2(0), T_{y_1}^{\Cb} \partial \Omega) = 0.
\end{align*}
But this contradicts Equation~\ref{eq:attractive}.
\end{proof}

\subsection{Finding hyperbolic automorphisms}

\begin{lemma}\label{lem:dual_cond}
Suppose $\Omega \subset \Cb^d$ is a bounded convex domain with $C^{1,\alpha}$ boundary, $\varphi_n \in \Aut(\Omega)$, and $p \in \Omega$. If $\varphi_n p \rightarrow x^+ \in \partial \Omega$, $\varphi_n^{-1} p \rightarrow x^- \in \partial \Omega$, and $T_{x^+}^{\Cb} \partial \Omega \neq T_{x^-}^{\Cb} \partial \Omega$, then $\varphi_n$ is hyperbolic for large $n$ and $H_{\varphi_n}^\pm \rightarrow T_{x^{\pm}}^{\Cb} \partial \Omega$. 
\end{lemma}

\begin{proof} Lemma~\ref{lem:para_inverse} and Lemma~\ref{lem:elliptic_inverse} imply that $\varphi_n$ is hyperbolic for large $n$ and Lemma~\ref{lem:cont_att} implies that $H_{\varphi_n}^\pm \rightarrow T_{x^{\pm}}^{\Cb} \partial \Omega$. 
\end{proof}

 \begin{proof}[Proof of Theorem~\ref{thm:exist_hyp}] Clearly if $\Aut(\Omega)$ contains a hyperbolic element then there exists $x,y \in \Lc(\Omega)$ so that $T_x^{\Cb} \partial \Omega \neq T_y^{\Cb} \partial \Omega$. 
 
 Next suppose that there exists $x^+,x^- \in \Lc(\Omega)$ so that $T_{x^+}^{\Cb} \partial \Omega \neq T_{x^-}^{\Cb} \partial \Omega$. Pick $\varphi_n, \phi_m \in \Aut(\Omega)$ and $p, q \in \Omega$ so that $\varphi_n p \rightarrow x^+$ and $\phi_m q \rightarrow x^-$. It is enough to consider the case where none of the $\varphi_n$ or $\phi_m$ are hyperbolic. 
 
 Pick relatively compact neighborhoods $U^{\pm} \subset \Cb^d$ of $T_{x^\pm}^{\Cb} \partial \Omega \cap \partial \Omega$ so that $\overline{U^+} \cap \overline{U^-} = \emptyset$. By Proposition~\ref{prop:uniform_att} there exists $m,n \geq 1$ so that $\varphi_n^{-1}(p) \in U^+$, $\varphi_n( U^- \cap \Omega) \subset U^+$,  $\phi_m^{-1}(p) \in U^-$, and $\phi_m( U^+ \cap \Omega) \subset U^-$. Now if $\gamma = \phi_m \varphi_n^{-1}$ we see that $\gamma(p) \in \phi_m(U^+ \cap \Omega) \subset U^-$ and $\gamma^{-1}(p) \in \varphi_n(U^- \cap \Omega) \subset U^+$. 
 
 But $U^+$ and $U^-$ were arbitrary relatively compact neighborhoods of $x^+$ and $x^-$ such that $\overline{U^+} \cap \overline{U^-} = \emptyset$, so we can find $m_k, n_k \rightarrow \infty$ so that if $\gamma_k = \phi_{m_k} \varphi_{n_k}^{-1}$ then $\gamma_k p \rightarrow x^-$ and $\gamma_k^{-1} p \rightarrow x^+$. Then for $k$ large $\gamma_k$ is hyperbolic by Lemma~\ref{lem:dual_cond}.
 \end{proof}
 
 \section{The behavior of hyperbolic elements}\label{sec:axial_autos}

In a non-positively curved metric spaces a hyperbolic isometry always translates a geodesic (see for instance~\cite[Chapter II.6 Theorem 6.8]{BH1999}). We now show that a similar phenomena holds for hyperbolic automorphisms. 

\begin{theorem}\label{thm:hyp_unif_conv}
Suppose $\Omega \subset \Cb^d$ is a bounded convex domain with $C^{1,\alpha}$ boundary and $\varphi \in \Aut(\Omega)$. If $\varphi$ is hyperbolic, then there exists $x^\pm \in H_{\varphi}^\pm$, $M \geq 0$, and $T \in \Rb$ so that 
\begin{align*}
\{ x^+ + e^{-2t} n_{x^+} : t > T \} \cup \{ x^- + e^{-2t} n_{x^-} : t > T \}  \subset \cup_{k \in \Zb} \ \varphi^k B_{\Omega}(o; M)
\end{align*}
where 
\begin{align*}
B_{\Omega}(o; M) = \{ p \in \Omega : K_\Omega(p,o) \leq M\}.
\end{align*}
\end{theorem}

\begin{remark} Recall from Proposition~\ref{prop:rough_geodesics} that curves of the form $t \rightarrow x+ e^{-2(t+T)}n_x$ are $K$-almost-geodesics, so the above theorem says that a hyperbolic automorphism almost translates an almost-geodesic.\end{remark}

\begin{proof}
Fix points $x^{\pm} \in \Lc(\Omega, \varphi^{\pm 1})$.

By Proposition~\ref{prop:rough_geodesics} there exists $K_0 \geq 1$ and $\epsilon >0$ so that for any $x \in \partial \Omega$ the curve $\sigma_{x}: \Rb_{\geq 0} \rightarrow \Omega$ given by 
\begin{align*}
\sigma_x(t) = x + \epsilon e^{-2t} n_{x}
\end{align*}
is an $K_0$-almost-geodesic. Now by Lemma~\ref{lem:dist} there exists $K_1 \geq 0$ so that 
\begin{align*}
K_{\Omega}(\sigma_{x^+}(t), \sigma_{x^-}(s)) \geq \frac{1}{2} \log \frac{1}{\delta_{\Omega}(\sigma_{x^+}(t))} +  \frac{1}{2} \log \frac{1}{\delta_{\Omega}(\sigma_{x^-}(s))}-K_1 \geq t+s- K_1
\end{align*}
for all $s,t \geq 0$. Next let $\gamma:[0, S] \rightarrow \Omega$ be a unit speed geodesic with $\gamma(0) = \sigma_{x^-}(0)$ and $\gamma(S) = \sigma_{x^+}(0)$.  Then define $\sigma: \Rb \rightarrow \Omega$ by 
\begin{align*}
\sigma(t) = \left\{ \begin{array}{ll} 
 \sigma_{x^-}(-t) & \text{ if } t \leq 0 \\
 \gamma(t) & \text{ if } 0 \leq t \leq S \\
\sigma_{x^+}(t-S) & \text{ if } S\geq t
\end{array} \right.
\end{align*} 
Then $\sigma$ will be an $K$-almost-geodesic for some $K \geq 1$. Moreover by construction 
\begin{align*}
\lim_{t \rightarrow \pm\infty} \sigma(t) = x^{\pm}.
\end{align*}

Now fix $o \in \Omega$. We claim that there exists an $M_0 > 0$ so that  
\begin{align*}
K_{\Omega}(\varphi^k o, \sigma) \leq M_0
\end{align*}
for all $k \in \Zb$. Suppose not, then for every $m \geq 0$ there exists $k_m \in \Zb$ so that 
\begin{align*}
K_{\Omega}(\varphi^{k_m} o, \sigma) \geq m.
\end{align*}
Now for each $m$, $\varphi^{-k_m} \sigma: \Rb \rightarrow \Omega$ is an $K$-almost-geodesic and by Lemma~\ref{lem:point_attract}
\begin{align*}
\lim_{t \rightarrow \pm \infty} d_{\Euc}( \varphi^{-k_m} \sigma(t), H^{\pm}_{\varphi}) = 0.
\end{align*}
Since $H^+_{\varphi} \neq H^-_{\varphi}$, by Theorem~\ref{thm:geod_converge}, we can pass to a subsequence and find $T_m \in \Rb$ so that the $K$-almost-geodesics $t \rightarrow (\varphi^{-k_m} \sigma)(t+T_m)$ converge to an $K$-almost-geodesic $\sigma_{\infty}: \Rb \rightarrow \Omega$. But then 
\begin{align*}
\infty 
&= \lim_{m \rightarrow \infty} K_{\Omega}(\varphi^{k_m} o, \sigma) \leq \lim_{m \rightarrow \infty} K_{\Omega}(\varphi^{k_m} o, \sigma(T_m)) \\
& = \lim_{m \rightarrow \infty} K_{\Omega}(o, \varphi^{-k_m} \sigma(T_m)) = K_{\Omega}(o, \sigma_{\infty}(0))
\end{align*}
which is a contradiction. Thus, there exists an $M_0 > 0$ so that  
\begin{align*}
K_{\Omega}(\varphi^k o, \sigma) \leq M_0
\end{align*}
for all $k \in \Zb$. 

For each $k \in \Zb$ let $t_k \in \Rb$ be such that $K_{\Omega}(\sigma(t_k), \varphi^k o) \leq M_0$. Then 
\begin{align*}
\lim_{k \rightarrow \pm \infty} t_k = \pm \infty
\end{align*}
and 
\begin{align*}
K_{\Omega}(\sigma(t_k), \sigma(t_{k+1})) \leq 2M_0 + K_{\Omega}(\varphi^k o, \varphi^{k+1}o ) = 2M_0 + K_{\Omega}(o, \varphi o ).
\end{align*}
Since 
\begin{align*}
K_{\Omega}(\sigma(t), \sigma(s)) \leq \abs{t-s} + K
\end{align*}
we see that 
\begin{align*}
\sigma(\Rb) \subset \cup_{k \in \Zb} \ \varphi^k B_{\Omega}(o; M)
\end{align*}
when
\begin{align*}
M = 2M_0 +\frac{1}{2} K_{\Omega}(o, \varphi o ) +\frac{3}{2}K.
\end{align*}
\end{proof}

  \section{Finding a limit point of finite type}\label{sec:greene_krantz}
  
  In this section we prove a special case of the Greene-Krantz conjecture: 

 \begin{theorem}\label{thm:unif_tang_conv}
Suppose $\Omega \subset \Cb^d$ is a bounded convex domain with $C^\infty$ boundary. If there exists $o \in \Omega$, $x \in \partial \Omega$, $M \geq 0$, and $T \in \Rb$ so that 
\begin{align*}
\{ x + e^{-2t} n_{x} : t > T \}  \subset \Aut(\Omega) \cdot B_{\Omega}(o; M),
\end{align*}
then $x$ has finite type in the sense of D'Angelo. 
\end{theorem}

The proof has two main steps: we first show that if $x$ had infinite type then we could use a rescaling argument to construct a holomorphic map $f:\Delta \times \Delta \rightarrow \Omega$ which is very close to being an isometric embedding (with respect to the Kobayashi metrics). The second step is to use the behavior of the Gromov product to show that such a holomorphic map cannot exist. 

\subsection{Rescaling} Motivated by language from real projective geometry (see for instance~\cite{B2008}), we say a convex set $\Omega \subset \Cb^d$ is \emph{$\Cb$-proper} if $\Omega$ does not contain any complex affine lines. By a theorem of Barth these are exactly the convex subsets of $\Cb^d$ for which the Kobayashi metric is non-degenerate (see Proposition~\ref{prop:completeness} above). 

Next let $\Xb_{d,0}$ be the set of pairs $(\Omega, p)$ where $\Omega \subset \Cb^d$ is an open $\Cb$-proper convex set and $p \in \Omega$. We then write $(\Omega_n, p_n) \rightarrow (\Omega, p)$ if $p_n \rightarrow p$ and $\Omega_n$ converges to $\Omega$ in the local Hausdorff topology. 

Frankel proved the following:

\begin{theorem}\cite[Theorem 5.6]{F1989}\label{thm:rescale} Suppose $\Omega$ is a $\Cb$-proper convex domain, $K \subset \Omega$ a compact subset, and $\varphi_n \in \Aut(\Omega)$. If there exists $k_n \in K$ and complex affine maps $A_n$ so that 
\begin{align*}
A_n(\Omega, \varphi_n k_n) \rightarrow (\wh{\Omega}, p)
\end{align*}
in $\Xb_{d,0}$, then $\Omega$ is biholomorphic to $\wh{\Omega}$. 
\end{theorem}

\subsection{Line type}

Given a function $f: \Cb \rightarrow \Rb$ with $f(0)=0$ let $\nu(f)$ denote the order of vanishing of $f$ at $0$. Suppose that $\Omega = \{ z \in \Cb^d : r(z) < 0\}$ where $r$ is a $C^\infty$ function with $\nabla r \neq 0$ near $\partial \Omega$. We say that a point $x \in \partial \Omega$ has \emph{finite line type $L$} if 
\begin{align*}
\sup \{ \nu( r \circ \ell) | \ell : \Cb \rightarrow \Cb^d \text{ is a non-trivial affine map and $\ell(0)=x$} \} = L.
\end{align*}
Notice that $\nu(r\circ \ell) \geq 2$ if and only if $\ell(\Cb)$ is tangent to $\Omega$. McNeal~\cite{M1992} proved that if $\Omega$ is convex  then $x \in \partial \Omega$ has finite line type if and only if it has finite type in the sense of D'Angelo (also see~\cite{BS1992}). In this paper, we say a convex domain $\Omega$ with $C^\infty$ boundary has \emph{finite line type $L$} if the line type of all $x \in \partial \Omega$ is at most $L$ and this bound is realized at some boundary point.

\subsection{Rescaling at a point of infinite type}

\begin{proposition}\label{prop:infinite_type_poly_disk}
Suppose $\Omega \subset \Cb^d$ is a bounded convex domain with $C^\infty$ boundary and $x \in \partial \Omega$ has infinite line type. Then there exists $t_n \rightarrow \infty$ and complex affine maps $A_n$ with the following properties:
\begin{enumerate}
\item $A_n(\Omega, x+e^{-t_n}n_{x}) \rightarrow (\wh{\Omega}, u)$ in $\Xb_{d,0}$, 
\item there exists a holomorphic map $f: \Delta \times \Delta \rightarrow \wh{\Omega}$ so that 
\begin{align*}
K_{\Delta}(z_1, z_2) = K_{\wh{\Omega}}(f(z_1,0), f(z_2,0))
\end{align*}
for all $z_1, z_2 \in \Delta$ and 
 \begin{align*}
\abs{t-s}-\log \sqrt{2} \leq  K_{\wh{\Omega}}(f(0,\tanh(t)), f(0,\tanh(s))) \leq \abs{t-s}.
\end{align*}
for every $s, t \geq 0$.
\end{enumerate}
 \end{proposition}
 
Proposition~\ref{prop:infinite_type_poly_disk} will follow from a series of lemmas, the first of which is due to Frankel. 

\begin{lemma}\cite[Theorem 9.3]{F1991}\label{lem:slices}
Suppose $\Omega \subset \Cb^d$ is a $\Cb$-proper convex domain. If $V \subset \Cb^d$ is a complex affine $k$-dimensional subspace intersecting $\Omega$, $p_n \in V \cap \Omega$, and $A_n \in \Aff(V)$ is a sequence of complex affine maps such that 
\begin{align*}
A_n(\Omega \cap V, p_n) \rightarrow (\wh{\Omega}_V, u) \text{ in } \Xb_{k,0},
\end{align*}
then there exists complex affine maps $B_n \in \Aff(\Cb^d)$ such that 
\begin{align*}
B_n(\Omega, p_n) \rightarrow (\wh{\Omega}, u) \text{ in } \Xb_{d,0}
\end{align*}
and $\wh{\Omega} \cap V = \wh{\Omega}_V$. 
\end{lemma}
 
\begin{lemma}\label{lem:infinite_type_rescale}
Suppose $\Omega \subset \Cb^2$ is a $\Cb$-proper convex domain with $0 \in \partial \Omega$ and 
\begin{align*}
\Omega \cap \Oc = \{(x+iy, z) \in \Oc :  y > f(x, z)\}
\end{align*}
where $\Oc$ is a neighborhood of the origin and $f:(\Rb \times \Cb) \cap \Oc \rightarrow \Rb$ is a convex non-negative function. Assume that 
\begin{align*}
\lim_{z \rightarrow 0} \frac{f(0,z)}{\abs{z}^n} = 0
\end{align*}
for all $n>0$ and the function $t \rightarrow f(t,0)$ is $C^1$ at $t=0$. Then there exists $t_n \rightarrow \infty$ and complex affine maps $A_n$ so that
\begin{align*}
A_n(\Omega, (ie^{-t_n},0)) \rightarrow (\wh{\Omega}, (i,0)) \text{ in } \Xb_{2,0}
\end{align*}
where
\begin{align*} 
\Hc \times \Delta \subset \wh{\Omega} \subset \{ (z_1, z_2) \in \Cb^2: \Imaginary(z_1) > 0\}
\end{align*}
and 
\begin{align*} 
\{ (z,1) : z \in \Cb\} \cap \wh{\Omega} = \emptyset.
\end{align*}
\end{lemma} 

\begin{proof}
The lemma follows from the proof of Proposition 6.1 in~\cite{Z2014} essentially verbatim, but for the readers convenience we will provide the argument.  

We can suppose that $\Oc = (V+iW) \times U$ where $V,W \subset \Rb$ and $U \subset \Cb$ are neighborhoods of $0$. By rescaling we may assume that $B_1(0) \subset U$. \newline

\noindent \textbf{Case 0:} Suppose that there exists $\delta >0$ so that $f(0,z) =0$ for $\abs{z} < \delta$. Then after a linear transformation in the second variable we can assume that 
\begin{align*}
\{0\} \times \Delta \subset F:= \partial \Omega \cap (\{0\} \times \Cb)
\end{align*}
and $(0,1) \in \partial F$. Now fix $t_n \rightarrow \infty$. Pick $z_n \in \Cb$ so that $(ie^{-t_n},z_n) \in \partial \Omega$ and 
\begin{align*}
\abs{z_n} = \inf\{ \abs{z} : (ie^{-t_n}, z) \in \partial \Omega\}.
\end{align*}
Next consider the linear maps 
\begin{align*}
A_n = \begin{pmatrix}
e^{t_n} & 0 \\
0 & z_n^{-1}
\end{pmatrix}.
\end{align*}
We claim that after passing to a subsequence 
\begin{align*}
A_n(\Omega, (ie^{-t_n},0)) \rightarrow (\wh{\Omega}, (i,0)) \text{ in } \Xb_{2,0}
\end{align*}
where $\wh{\Omega} \subset \Cb^2$ is a $\Cb$-proper convex set satisfying the conclusion of the the lemma. By passing to a subsequence we can assume that $\overline{A_n \Omega}$ converges to a closed convex set $C$ in the local Hausdorff topology. By construction 
\begin{align*}
C \subset \{ (z_1, z_2) : \Imaginary(z_1) \geq 0\}.
\end{align*}
Since the function $t \rightarrow f(t,0)$ is $C^1$ at $t=0$ we see that $\Hc \times \{ 0\} \subset C$. Since $(0,1) \in \partial F$ we also see that $\lim_{n \rightarrow \infty} \abs{z_n} = 1$. So 
\begin{align*}
\{0\} \times \Delta \subset  \partial C.
\end{align*}
Then by convexity $\Hc \times \Delta \subset C$. Thus $C$ has non-empty interior. Let $\wh{\Omega}$ be the interior of $C$. Then $A_n \Omega$ converges to $\wh{\Omega}$ in the local Hausdorff topology. By the remarks above 
\begin{align*}
\Hc \times \Delta \subset \wh{\Omega} \subset \{ (z_1, z_2) : \Imaginary(z_1) > 0\}.
\end{align*}

We next claim that 
\begin{align*} 
\wh{\Omega} \cap (\Cb \times \{1 \} )= \emptyset.
\end{align*}
By construction $(0,1) \in \partial \wh{\Omega}$. Then, by convexity, 
\begin{align*}
\Hc \times \{1\} = (0,1) + \Hc \times \{0\}  \subset \wh{\Omega} \cup \partial \wh{\Omega}.
\end{align*}
But since $\wh{\Omega}$ is open and convex either 
\begin{align*}
\Hc\times \{1\} \subset  \wh{\Omega} \text{ or }  \Hc \times \{1\} \subset \partial \wh{\Omega}.
\end{align*}
Since $(i,1) \in \partial \wh{\Omega}$ we must have
\begin{align*}
 \Hc \times \{1\} \subset \partial \wh{\Omega}.
\end{align*}
Which in turn implies that
\begin{align*}
\wh{\Omega} \cap (\Cb \times \{1 \} )= \emptyset.
\end{align*}

We can now show that $\wh{\Omega}$ is $\Cb$-proper. Suppose that an affine map $z \rightarrow (a_1,a_2)z+(b_1,b_2)$ has image in $\wh{\Omega}$. Since 
\begin{align*}
\wh{\Omega} \subset \{ (z_1, z_2) \in \Cb^2 : \Imaginary(z_1) > 0\}
\end{align*}
 we see that $a_1=0$. And since $\wh{\Omega} \cap (\Cb \times \{1 \} )= \emptyset$ we also see that $a_2=0$. So $\wh{\Omega}$ does not contain any non-trivial complex affine lines and hence is $\Cb$-proper. This completes the argument in Case 0. \newline

\noindent \textbf{Case 1:}  Suppose for any $\delta > 0$ there exists $z \in \Cb$ with $\abs{z} < \delta$ and  $f(0,z) \neq 0$. Since 
\begin{align*}
\lim_{z \rightarrow 0} \frac{ f(0,z)}{\abs{z}^n} = 0,
\end{align*}
we can find $a_n \searrow 0$ and $z_n \in B_1(0)$ such that $f(0,z_n) = a_n \abs{z_n}^n$ and for all $w \in \Cb$ with $\abs{w} \leq \abs{z_n}$ we have 
\begin{align*}
f(0,w) \leq a_n \abs{w}^n.
\end{align*}
By the hypothesis of case 1 we see that $z_n \rightarrow 0$ and hence $f(0,z_n) \rightarrow 0$. So by passing to a subsequence we may assume that $\abs{f(0,z_n)} < 1$.

Consider the linear transformations
\begin{align*}
A_n= \begin{pmatrix} \frac{1}{f(0,z_n)}  & 0 \\ 0 & z_n^{-1} \end{pmatrix} \in \GL(\Cb^2)
\end{align*}
and let $\Omega_n = A_n \Omega$. By passing to a subsequence we can assume that $\overline{A_n \Omega}$ converges to a closed convex set $C$ in the local Hausdorff topology. By construction 
\begin{align*}
C \subset \{ (z_1, z_2) : \Imaginary(z_1) \geq 0\}.
\end{align*}
Since the function $t \rightarrow f(t,0)$ is $C^1$ at $t=0$ we see that $\Hc \times \{ 0\} \subset C$. 

We next show that $\{0\} \times \Delta \subset \partial C$. If $\Oc_n = A_n \Oc$ we have
\begin{align*}
 A_n\Omega_n \cap \Oc_n = \{ (x+iy,z ) : x \in V_n, z \in U_n,  y > f_n(x,z)\}
\end{align*}
where $V_n = f(z_n,0)^{-1}V$, $U_n = z_n^{-1}U$, and
\begin{align*}
f_n(x,z) =  \frac{1}{f(0,z_n)} f\left( f(0,z_n)x, z_n z\right).
\end{align*}
For $\abs{w} < 1$ we then have
\begin{align*}
f_n(0,w) =\frac{f\left( 0, z_n w\right)}{f(0, z_n)} \leq \frac{ a_n \abs{z_n}^n \abs{w}^n }{f(0,z_n)} = \abs{w}^n
\end{align*}
which implies that
\begin{align*}
\{0\} \times \Delta \subset \partial C.
\end{align*}

Then by convexity $\Hc \times \Delta \subset C$. Thus $C$ has non-empty interior. Let $\wh{\Omega}$ be the interior of $C$. Then $A_n \Omega$ converges to $\wh{\Omega}$ in the local Hausdorff topology. By the remarks above 
\begin{align*}
\Hc \times \Delta \subset \wh{\Omega} \subset \{ (z_1, z_2) : \Imaginary(z_1) > 0\}.
\end{align*}

Since $f_n(0,1)=1$ we see that $(i,1) \in \partial \Omega_n$ for all $n$ and so $(i,1) \in \partial \wh{\Omega}$. Then following the argument in Case 0 we see that 
\begin{align*}
\wh{\Omega} \cap (\Cb \times \{1 \})= \emptyset
\end{align*}
and $\wh{\Omega}$ is $\Cb$-proper. 
\end{proof}

\begin{lemma}\label{lem:half_plane_slice} Suppose that $\Omega \subset \Cb^d$ is a $\Cb$-proper convex domain. If 
\begin{align*}
\Omega \cap (\Cb \times \{(0, \dots, 0)\}) = \Hc \times \{(0,\dots, 0)\},
\end{align*}
then the map $f: \Hc \rightarrow \Omega$ given by $f(z) = (z,0,\dots, 0)$ induces an isometric embedding $(\Hc, K_{\Hc}) \rightarrow (\Omega, K_\Omega)$. 
\end{lemma}

\begin{proof}
The distance decreasing property of the Kobayashi metric implies that 
\begin{align*}
K_{\Omega}(f(z_1), f(z_2)) \leq K_{\Hc}(z_1, z_2).
\end{align*} 

For the opposite inequality, let $H_{\Rb}$ 
be a real hyperplane so that $\Rb \times \{0\} \subset H_{\Rb}$ and $H_{\Rb} \cap \Omega = \emptyset$. 
Then there exists a linear map $A:\Cb^d \rightarrow \Cb^d$ so that 
\begin{align*}
A(z, 0,\dots, 0) = (z,0, \dots, 0)
\end{align*}
and 
\begin{align*}
A(H_{\Rb}) = \{ (z_1, \dots, z_d) \in \Cb^d : \Imaginary(z_1) =0\}.
\end{align*}
Consider the map $P: \Cb^d \rightarrow \Cb$ given by $P(z_1, \dots, z_d) = z_1$ and the map $F := P \circ A$. Then $F(\Omega) = \Hc$ and $F \circ f = \Id$. So we see that $K_{\Omega}(f(z_1), f(z_2)) \geq K_{\Hc}(z_1, z_2)$. 
\end{proof}
 
 \begin{proof}[Proof of Proposition~\ref{prop:infinite_type_poly_disk}]
By Lemma~\ref{lem:slices} and Lemma~ \ref{lem:infinite_type_rescale} we can find $t_n \rightarrow \infty$ and affine maps $A_n$ so that
\begin{align*}
A_n(\Omega, x+e^{-t_n}n_{x}) \rightarrow (\wh{\Omega}, u) \text{ in } \Xb_{d,0}
\end{align*}
where
\begin{align*} 
\Hc \times \Delta \subset \wh{\Omega} \cap (\Cb^2 \times \{ (0, \dots, 0)\}) \subset \{ (z_1, z_2) \in \Cb^2 : \Imaginary(z_1) > 0\}
\end{align*}
and 
\begin{align*} 
\left(\Cb \times \{(1,0,\dots,0)\}\right) \cap \wh{\Omega} = \emptyset.
\end{align*}
Then
\begin{align*} 
\Hc \times \{(1, 0, \dots, 0)\}  = \wh{\Omega} \cap (\Cb \times \{ (0, \dots, 0)\}).
\end{align*}
Consider the map $f:\Delta \times \Delta \rightarrow \wh{\Omega}$ given by 
\begin{align*}
f(z,w) = \left( i\frac{1+z}{1-z}, w, 0, \dots, 0\right).
\end{align*}
By Lemma~\ref{lem:half_plane_slice}
\begin{align*}
K_{\wh{\Omega}}\Big( (z_1, 0, \dots, 0), (z_2, 0, \dots, 0)\Big) = K_{\Hc}(z_1, z_2).
\end{align*}
So
\begin{align*}
K_{\Delta}(z_1, z_2) = K_{\wh{\Omega}}(f(z_1,0), f(z_2,0))
\end{align*}
for all $z_1, z_2 \in \Delta$. 

Let $x_t = f(0,\tanh(t)) = (i,\tanh(t), 0, \dots, 0)$. Since $\{i\} \times \Delta \times \{ (0, \dots, 0) \} \subset \wh{\Omega}$ we see that 
 \begin{align*}
K_{\wh{\Omega}}\Big(x_t,x_s\Big) \leq K_\Delta(\tanh(t), \tanh(s)) = \abs{t-s}.
\end{align*}
On the other hand, the complex line $L=\{ (z,1,0,\dots, 0) : z \in \Cb\}$ does not intersect $\wh{\Omega}$ and so $(i,1,0,\dots, 0) \notin \wh{\Omega}$. Then by Lemma 2.6 in~\cite{Z2014} we have 
 \begin{align*}
K_{\wh{\Omega}}\Big(x_t,x_s\Big) \geq \frac{1}{2}\abs{ \log \frac{ \norm{x_t-(i,1,0,\dots, 0)} }{ \norm{x_s-(i,1,0,\dots, 0)} }}= \frac{1}{2}\abs{ \log \frac{ 1- \tanh(t)}{1-\tanh(s)} }.
\end{align*}
Since 
\begin{align*}
\tanh(x) = 1 - \frac{1}{e^{2x}+1}
\end{align*}
when $x \in \Rb$, we see that 
 \begin{align*}
K_{\wh{\Omega}}\Big(x_t,x_s\Big) \geq \abs{t-s} - \frac{1}{2} \log(2)
\end{align*}
for $t,s \geq 0$. 
 \end{proof}
 
 \subsection{Non-existence of certain holomorphic maps} 
 
  \begin{theorem}\label{thm:prod_map}
 Suppose $\Omega \subset \Cb^d$ is a bounded convex domain with $C^{1,\alpha}$ boundary. Then there does not exist a holomorphic map $f: \Delta \times \Delta \rightarrow \Omega$ and $\kappa \geq 0$ so that
 \begin{enumerate}
 \item for every $z_1, z_2 \in \Delta$
 \begin{align*}
K_{\Delta}(z_1,z_2) - \kappa \leq K_{\Omega}\Big(f(z_1,0), f(z_2,0)\Big) \leq  K_{\Delta}(z_1,z_2) + \kappa,
 \end{align*}
 \item for every $s, t \geq 0$
 \begin{align*}
\abs{t-s} - \kappa \leq  K_{\Omega}\Big(f(0, \tanh(t)), f(0, \tanh(s))\Big) \leq \abs{t-s} + \kappa.
\end{align*}
\end{enumerate}
\end{theorem}

\begin{remark} If $f: \Delta \times \Delta \rightarrow \Omega$ is holomorphic and induces an isometric embedding $(\Delta \times \Delta, K_{\Delta \times \Delta}) \rightarrow (\Omega, K_{\Omega})$ then for every $s, t \in \Rb$
 \begin{align*}
K_{\Omega}\Big(f(0, \tanh(t)), f(0, \tanh(s))\Big) = \abs{t-s}.
\end{align*}
So Theorem~\ref{thm:prod_map} implies Theorem~\ref{thm:prod_map_i}.
\end{remark}

\begin{proof} Let $o = f(0,0)$. For $e^{i\theta} \in \partial \Delta$ let 
\begin{align*}
\sigma_{\theta}(t) := f(\tanh(t)e^{i\theta},0).
\end{align*}
Now for $s,t \in \Rb$
\begin{align*}
K_{\Delta}( \tanh(t)e^{i\theta}, \tanh(s)e^{i\theta}) = \abs{t-s}
\end{align*}
and so 
\begin{align*}
\lim_{s,t \rightarrow \infty} (\sigma_\theta(t)|\sigma_\theta(s))_{o} \geq \frac{1}{2} \left( \lim_{s,t \rightarrow \infty}  \min\{t,s\} -3\kappa\right)= \infty.
\end{align*}
Thus by Theorem~\ref{thm:gromov_prod} for each $e^{i\theta} \in \partial \Delta$ there exists $x_\theta \in \partial \Omega$ so that 
\begin{align*}
\lim_{t \rightarrow \infty} d_{\Euc}(\sigma_{\theta}(t), T_{x_{\theta}}^{\Cb} \partial \Omega) = 0.
\end{align*}

Next let $\sigma(t) := f(0, \tanh(t))$. Then 
\begin{align*}
\lim_{s,t \rightarrow \infty} (\sigma(t)|\sigma(s))_{o} \geq\frac{1}{2} \left( \lim_{s,t \rightarrow \infty}  \min\{t,s\} -3\kappa\right)= \infty.
\end{align*}
Thus by Theorem~\ref{thm:gromov_prod} there exists $x \in \partial \Omega$ so that 
\begin{align*}
\lim_{t \rightarrow \infty} d_{\Euc}(\sigma(t), T_{x}^{\Cb} \partial \Omega) = 0.
\end{align*}

Now for $e^{i\theta} \in \partial \Delta$ and $t,s \geq 0$
\begin{align*}
 (\sigma_\theta(t)|\sigma(s))_{o} \geq \frac{1}{2} \left( t + s -2\kappa - K_{\Omega}(\sigma_{\theta}(t), \sigma(s))\right)
\end{align*}
and 
\begin{align*}
K_{\Omega}(\sigma_{\theta}(t), \sigma(s))& \leq K_{\Delta \times \Delta}( (\tanh(t) e^{i\theta},0), (0,\tanh(s)) ) \\
&= \max \{ K_{\Delta}(\tanh(t)e^{i\theta},0), K_{\Delta}(0, \tanh(s))\} = \max\{ t,s\}.
\end{align*}
Thus
\begin{align*}
\lim_{s,t \rightarrow \infty} (\sigma_\theta(t)|\sigma(s))_{o} \geq \frac{1}{2} \left( \lim_{s,t \rightarrow \infty} \min\{t,s\}-2\kappa \right)= \infty.
\end{align*}
So by Theorem~\ref{thm:gromov_prod}, $x_{\theta} \in T_{x}^{\Cb} \partial \Omega$ which implies that $T_{x_\theta}^{\Cb} \partial \Omega = T_{x}^{\Cb} \partial \Omega$.

Now by translating and rotating $\Omega$ we may assume that 
\begin{align*}
T_{x}^{\Cb} \partial \Omega = \{ (z_1, \dots, z_d) \in \Cb^d : z_1=0\}
\end{align*}
and 
\begin{align*}
\Omega \subset  \{ (z_1, \dots, z_d) \in \Cb^d : \Imaginary z_1> 0\}.
\end{align*}
Next consider the projection $P:\Cb^d \rightarrow \Cb$ given by $P(z_1, \dots, z_d) = z_1$ and the holomorphic function $g : \Delta \rightarrow \Cb$ given by 
\begin{align*}
g(z) = P(f(z,0)).
\end{align*}
Then $\Imaginary g(z) > 0$ for all $z \in \Delta$, $g$ is bounded, and 
\begin{align*}
\lim_{r \rightarrow 1} g(r e^{i\theta}) = 0
\end{align*}
for any $\theta \in \Rb$. But this is impossible by the Cauchy integral formula and the dominated convergence theorem. 
\end{proof}

\subsection{Proof of Theorem~\ref{thm:unif_tang_conv}} We can now prove Theorem~\ref{thm:unif_tang_conv}. Suppose that $\Omega \subset \Cb^d$ is a bounded convex domain with $C^\infty$ boundary. Assume $o \in \Omega$, $x \in \partial \Omega$, $M \geq 0$, and $T \in \Rb$ are so that 
\begin{align*}
\{ x+ e^{-2t} n_{x} : t > T \}  \subset \Aut(\Omega) \cdot B_{\Omega}(o; M).
\end{align*}
Suppose for a contradiction that $x$ has infinite type. Then we can find $t_n \rightarrow \infty$ and affine maps $A_n$ such that
\begin{align*}
A_n(\Omega, x+e^{-t_n}n_{x}) \rightarrow (\wh{\Omega}, u) \text{ in } \Xb_{d,0}
\end{align*}
and there exists a map $f: \Delta \times \Delta \rightarrow \wh{\Omega}$ with the properties in Proposition~\ref{prop:infinite_type_poly_disk}. Now by Theorem~\ref{thm:rescale} the domain $\wh{\Omega}$ is biholomorphic to $\Omega$. But this is impossible by Theorem~\ref{thm:prod_map}.

\section{The entire boundary has finite type}\label{sec:entire_bd_finite_type}

In this section we prove:

\begin{proposition}\label{prop:finite_type}
Suppose $\Omega \subset \Cb^d$ is a bounded convex domain with $C^\infty$ boundary. If there exists $x \in \partial \Omega$ with finite line type, $o \in \Omega$, and $\varphi_n \in \Aut(\Omega)$ so that $\varphi_n o \rightarrow x$ non-tangentially, then $\partial \Omega$ has finite line type. 
\end{proposition}

The idea will be to first use a scaling argument to show that $\Omega$ is biholomorphic to a domain of the form 
\begin{align*}
\wh{\Omega} = \{ (z_1, \dots, z_d) \in \Cb^d : \Imaginary(z_1) > P(z_2, \dots, z_d) \}
\end{align*}
where $P: \Cb^{d-1} \rightarrow \Rb$ is non-degenerate non-negative convex polynomial and there exists  $\delta_1, \dots \delta_{d} \in (0,1/2)$ so that 
\begin{align*}
P( t^{\delta_1} z_1, \dots, t^{\delta_{d} }z_d) = t P(z_1, \dots, z_d)
\end{align*}
for all $t \geq 0$. By Theorem 1.7 in~\cite{Z2014b} the metric space $(\wh{\Omega}, K_{\wh{\Omega}})$ is Gromov hyperbolic. But then $(\Omega, K_\Omega)$ is Gromov hyperbolic and so $\partial \Omega$ has finite line type by Theorem 1.1 in~\cite{Z2014}.

Before starting the proof of the proposition we will need two lemmas:

\begin{lemma}\label{lem:dist_to_normal_line}
Suppose $\Omega \subset \Cb^d$ is a bounded domain with $C^1$ boundary. If $x \in \partial \Omega$ and $p_n \in \Omega$ converges to $x$ non-tangentially, then 
\begin{align*}
\lim_{n \rightarrow \infty} K_{\Omega}(p_n, N_x) = 0
\end{align*}
where $N_x = \Omega \cap (x + \Rb_{\geq 0} n_x)$.
\end{lemma}

\begin{proof}
This follows immediately from the estimate 
\begin{align*}
k_{\Omega}(p;v) \leq \frac{\norm{v}}{\delta_{\Omega}(p;v)}
\end{align*}
on the Kobayashi metric. 
\end{proof}

We say a polynomial $P:\Cb^d \rightarrow \Rb$ is \emph{non-degenerate} if the set $\{P(z)=0\}$ contains no complex affine lines. From the proof of Proposition 2 in~\cite{Y1992} one has the following:

\begin{lemma}
Suppose that $f:\Cb^d \rightarrow \Rb$ is a $C^\infty$ non-negative convex function so that $f(0)=0$ and for every $v \in \Cb^d$ the function $z \in \Cb \rightarrow r(zv) \in \Rb$ does not vanish to infinite order at $z=0$. Then after making a linear change of coordinates there exists $\delta_1, \dots \delta_{d} \in (0,1/2)$ so that 
\begin{align*}
P(z_1, \dots, z_d) = \lim_{ t \rightarrow 0} \frac{1}{t} f( t^{\delta_1} z_1, \dots, t^{\delta_{d} }z_d) 
\end{align*}
where $P : \Cb^d \rightarrow \Rb$ is a non-degenerate non-negative convex polynomial and the convergence is in the $C^\infty$ topology. Moreover, 
\begin{align*}
P( t^{\delta_1} z_1, \dots, t^{\delta_{d} }z_d) = t P(z_1, \dots, z_d)
\end{align*}
for all $t \geq 0$. 
\end{lemma}

\begin{proof}[Proof of Proposition~\ref{prop:finite_type}]
By applying an affine transformation we can assume that $x=0$ and $T_0 \partial \Omega = \Rb \times \Cb^{d-1}$. Now there exists neighborhoods $U,V \subset \Rb$ of $0$ and a neighborhood $W \subset \Cb^{d-1}$ of $0$ and a $C^\infty$ function $f: V \times W \rightarrow \Rb$ so that 
\begin{align*}
\Omega \cap \Oc = \{ (x+iy, z) \in \Oc : y >f(x, z) \}
\end{align*}
where $\Oc = (V + iU) \times W$. By the above lemma we can make an linear change of coordinates in the last $d-1$ variables and find $\delta_1, \dots \delta_{d-1} \in (0,1/2)$ so that 
\begin{align*}
\frac{1}{t} f(0, t^{\delta_1} z_1, \dots, t^{\delta_{d-1} }z_d) 
\end{align*}
converges locally uniformly in the $C^\infty$ topology to a non-degenerate non-negative convex polynomial $P: \Cb^{d-1} \rightarrow \Rb$. Moreover, 
\begin{align*}
P( t^{\delta_1} z_1, \dots, t^{\delta_{d-1} }z_d) = t P(z_1, \dots, z_d)
\end{align*}
for all $t \geq 0$.

Now suppose that $\varphi_n o \rightarrow 0$ non-tangentially. Then by Lemma~\ref{lem:dist_to_normal_line} there exists a compact set $K \subset \Omega$, elements $k_n \in K$, and a sequence $t_n \rightarrow \infty$ so that 
\begin{align*}
\varphi_n k_n = x + e^{-2t_n} n_x = (e^{-2t_n}, 0, \dots, 0).
\end{align*}
Consider the linear maps
\begin{align*}
\Lambda_n : = \begin{pmatrix} e^{2t_n} & & & \\ & e^{2t_n/\delta_1} & & \\ & & \ddots & \\ & & & e^{2t_n/\delta_{d-1}} \end{pmatrix}.
\end{align*}
Notice that $\Lambda_n \varphi_n k_n = (1, 0, \dots, 0)$. Moreover $\Lambda_n \Omega$ converges in the local Hausdorff topology to the domain 
\begin{align*}
\wh{\Omega} = \{ (z_1, \dots, z_d) \in \Cb^d : \Imaginary(z_1) > P(z_2, \dots, z_d) \}.
\end{align*}
Thus $\Omega$ is biholomorphic to $\wh{\Omega}$ by Theorem~\ref{thm:rescale}. Then by Theorem 1.7 in~\cite{Z2014b} the metric space $(\wh{\Omega}, K_{\wh{\Omega}})$ is Gromov hyperbolic. But then $(\Omega, K_\Omega)$ is Gromov hyperbolic and so $\partial \Omega$ has finite line type by Theorem 1.1 in~\cite{Z2014}.
\end{proof}

 \section{Proof of Theorem~\ref{thm:main}}
 
First suppose $\Omega \subset \Cb^d$ is a bounded convex domain with $C^\infty$ boundary and there exists $x, y \in \Lc(\Omega)$ with $T_{x}^{\Cb} \partial \Omega \neq T_{y}^{\Cb} \partial \Omega$. Then $\Aut(\Omega)$ contains a hyperbolic element $\varphi$ by Theorem~\ref{thm:exist_hyp}. Then by Theorem~\ref{thm:hyp_unif_conv} there exists $z \in H_{\varphi}^+$, $M \geq 0$, and $T \in \Rb$ so that 
\begin{align*}
\{ z + e^{-2t} n_{z} : t > T \} \subset \cup_{k \in \Zb} \  \varphi^k B_{\Omega}(o; M).
\end{align*}
Now by Theorem~\ref{thm:unif_tang_conv} the point $z \in \partial \Omega$ has finite type. Then by Proposition~\ref{prop:finite_type} the entire boundary has finite type. Then by Bedford and Pinchuk's characterization of polynomial ellipsoids (Theorem~\ref{thm:BP} in the introduction) we see that $\Omega$ is biholomorphic to a polynomial ellipsoid.

Next suppose that $\Omega$ is biholomorphic to a domain
\begin{align*}
\Cc: = \{ (w,z) \in \Cb \times \Cb^{d-1} : \Imaginary(w) > p(z)\}
\end{align*}
where $p:\Cb^{d-1} \rightarrow \Rb$ is a weighted homogeneous balanced polynomial.

Notice that 
\begin{align*}
\Cc \cap (\Cb \times \{0\} ) = \Hc \times \{0\}.
\end{align*}
Moreover $\Aut(\Cc)$ contains real translations in the first variable and a dilation. Thus
\begin{align*}
\Hc \times \{0\} \subset \Aut(\Cc) \cdot (i,0,\dots, 0). 
\end{align*}
So there exist a proper holomorphic map $\varphi: \Delta \rightarrow \Omega$ so that $\varphi(\Delta) \subset \Aut(\Omega) \cdot \varphi(0)$. This implies that 
\begin{align*}
\overline{\varphi(\Delta)} \setminus \varphi(\Delta) \subset \Lc(\Omega).
\end{align*}

Now suppose for a contradiction that $\Lc(\Omega)$ is contained in a single closed complex face of $\Omega$. Then there exists an $x \in \partial \Omega$ so that 
\begin{align*}
\lim_{r \rightarrow 1} d_{\Euc}(\varphi(re^{i\theta}), T_x^{\Cb} \partial \Omega) = 0
\end{align*}
for all $\theta \in \Rb$. Now by translating and rotating $\Omega$ we may assume that 
\begin{align*}
T_{x}^{\Cb} \partial \Omega = \{ (z_1, \dots, z_d) \in \Cb^d : z_1=0\}
\end{align*}
and 
\begin{align*}
\Omega \subset  \{ (z_1, \dots, z_d) \in \Cb^d : \Imaginary z_1> 0\}.
\end{align*}
Next consider the projection $P:\Cb^d \rightarrow \Cb$ given by $P(z_1, \dots, z_d) = z_1$ and the holomorphic function $g : \Delta \rightarrow \Cb$ given by 
\begin{align*}
g(z) = P(\varphi(z)).
\end{align*}
Then $\Imaginary g(z) > 0$ for all $z \in \Delta$, $g$ is bounded, and 
\begin{align*}
\lim_{r \rightarrow 1} g(r e^{i\theta}) = 0
\end{align*}
for any $\theta \in \Rb$. But this is impossible by the Cauchy integral formula and the dominated convergence theorem. 

\section{Boundary extensions of isometric embeddings}

We now apply Theorem~\ref{thm:gromov_prod} to prove the following boundary extension theorem:

\begin{theorem}\label{thm:cont_ext} Suppose $\Omega_1 \subset \Cb^{d_1}$ and $\Omega_2 \subset \Cb^{d_2}$ are bounded convex domains with $C^{1,\alpha}$ boundaries. If $\Omega_2$ is $\Cb$-strictly convex, then every isometric embedding $f:(\Omega_1,K_{\Omega_1}) \rightarrow (\Omega_2,K_{\Omega_2})$ extends to a  continuous map $\overline{f}: \overline{\Omega_1} \rightarrow \overline{\Omega_2}$. 
\end{theorem}

\begin{proof}
We first show that  $\lim_{z \rightarrow x} f(z)$ exists when $x \in \partial \Omega_1$. Since $f$ is an isometric embedding we see that 
\begin{align*}
(z|w)_o = (f(z)|f(w))_{f(o)}
\end{align*}
for any $z,w,o \in \Omega_1$. So if $z_n \rightarrow x$, $w_n \rightarrow x$, $f(z_n) \rightarrow y_1$, and $f(w_n) \rightarrow y_2$ then
\begin{align*}
\lim_{n \rightarrow \infty} (f(z_n)|f(w_n))_{f(o)} = \lim_{n \rightarrow \infty} (z_n|w_n)_o = \infty
\end{align*}
by part (1) of Theorem~\ref{thm:gromov_prod}. Thus by part (2) of Theorem~\ref{thm:gromov_prod} we see that $T_{y_1}^{\Cb} \partial \Omega = T_{y_2}^{\Cb} \partial \Omega$. But then, since $\Omega_2$ is $\Cb$-strictly convex, we see that $y_1= y_2$. Thus implies that $\lim_{z \rightarrow x} f(z)$ exists. 

Then define $\overline{f}: \overline{\Omega}_1 \rightarrow \overline{\Omega}_2$ by
\begin{align*}
\overline{f}(x) := \left\{ \begin{array}{ll} 
f(x) & \text{ if } x \in \Omega_1 \\ 
\lim_{z \rightarrow x} f(z) & \text{ if } x \in \partial \Omega_1
\end{array} \right.
\end{align*}
We claim that $\overline{f}$ is continuous. So suppose that $z_n \rightarrow z$ in $\overline{\Omega}_1$. If $z \in \Omega_1$ then  
\begin{align*}
\overline{f}(z_n) = f(z_n) \rightarrow f(z)= \overline{f}(z).
\end{align*}
So assume that $z \in \partial \Omega_1$. Then, to avoid cases, approximate $z_n$ by $z_n^\prime \in \Omega_1$ so that 
\begin{align*}
\max\left\{ d_{\Euc}(z_n, z_n^\prime), d_{\Euc}( \overline{f}(z_n), f(z_n^\prime))\right\} < 1/n.
\end{align*}
Then $z_n^\prime \rightarrow z$ so (by definition of $\overline{f}$)
\begin{align*}
\lim_{n \rightarrow \infty} f(z_n^\prime) = \overline{f}(z).
\end{align*}
But by construction 
\begin{align*}
\lim_{n \rightarrow \infty} f(z_n^\prime) = \lim_{n \rightarrow \infty} \overline{f}(z_n).
\end{align*}
So $\overline{f}(z_n) \rightarrow \overline{f}(z)$ and thus $\overline{f}$ is continuous. 
\end{proof}

\bibliographystyle{alpha}
\bibliography{complex_kob}

\end{document}